\documentclass[12pt]{amsart}
\usepackage{amsmath}
\usepackage{amsthm}
\usepackage{amsfonts}
\usepackage{amssymb}
\usepackage{tikz-cd}
\usepackage{graphicx}
\usepackage{hyperref}
\usepackage{url}
\usepackage{float}
\usepackage{comment}
\usepackage{lineno}

\newcommand{\meet}{\land}
\newcommand{\join}{\lor}
\newcommand{\zer}{Z_\ell}
\newcommand{\cozer}{\func{coz}_\ell}
\newcommand{\interi}{\func{int}}

\def\bal{\boldsymbol{\mathit{ba}\ell}}
\def\ubal{\boldsymbol{\mathit{uba}\ell}}
\def\cubal{\boldsymbol{\mathit{cuba}\ell}}

\def\mbal{\boldsymbol{\mathit{mba}\ell}}
\def\mubal{\boldsymbol{\mathit{muba}\ell}}

\def\mcubal{\boldsymbol{\mathit{mcuba}\ell}}

\def\bav{\boldsymbol{\mathit{bav}}}
\def\ubav{\boldsymbol{\mathit{ubav}}}
\def\mbav{\boldsymbol{\mathit{mbav}}}
\def\mubav{\boldsymbol{\mathit{mubav}}}

\def\ma{\sf{MA}}

\def\ba{\sf{BA}}

\newcommand\Stone{{\sf Stone}}
\newcommand\KHaus{{\sf KHaus}}

\newcommand\K{\sf{KF}}
\newcommand\DF{\sf{DF}}

\newcommand\KHK{\sf{KHF}}

\newcommand\Clop{{\sf Clop}}
\newcommand\Id{\func{Id}}

\newcommand\FA{\mathfrak{A}}

\newtheorem{theorem}{Theorem}[section]
\newtheorem*{theorem*}{Theorem}
\newtheorem{lemma}[theorem]{Lemma}
\newtheorem*{lemma*}{Lemma}
\newtheorem{proposition}[theorem]{Proposition}
\newtheorem*{proposition*}{Proposition}

\newtheorem*{corollary*}{Corollary}

\theoremstyle{definition}
\newtheorem{definition}[theorem]{Definition}
\newtheorem{convention}[theorem]{Convention}
\newtheorem{notation}[theorem]{Notation}
\newtheorem*{definition*}{Definition}

\newtheorem*{example*}{Example}

\newtheorem{remark}[theorem]{Remark}
\newtheorem*{remark*}{Remark}

\newcommand{\func}[1]{\operatorname{#1}}

\newcommand\xnot{\mathpalette\xxnot}
\newcommand*\xxnot[2]{\mathrel{\ooalign{%
  \hfil$#1\not\mathrel{\hphantom=}$\hfil\cr\hfil$#1#2$\hfil\cr}}}

\title{Modal Operators on Rings of Continuous Functions}

\author{G.~Bezhanishvili, L.~Carai, P.~J.~Morandi}
\date{}

\subjclass[2010]{03B45; 54C30; 06F25; 06E25; 06E15}
\keywords{modal algebra, Kripke frame, real-valued function, $\ell$-algebra, compact Hausdorff space, continuous relation}

\begin{document}

\begin{abstract}
It is a classic result in modal logic that the category of modal algebras is dually equivalent to the category of descriptive frames. The latter are Kripke frames equipped with a Stone topology such that the binary relation is continuous. This duality generalizes the celebrated Stone duality. Our goal is to further generalize descriptive frames so that the topology is an arbitrary compact Hausdorff topology. For this, instead of working with the boolean algebra of clopen subsets of a Stone space, we work with the ring of continuous real-valued functions on a compact Hausdorff space. The main novelty is to define a modal operator on such a ring utilizing a continuous relation on a compact Hausdorff space.

Our starting point is the well-known Gelfand duality between the category $\KHaus$ of compact Hausdorff spaces and the category $\ubal$ of uniformly complete bounded archimedean $\ell$-algebras. We endow a bounded archimedean $\ell$-algebra with a modal operator, which results in the category $\mbal$ of modal bounded archimedean $\ell$-algebras. Our main result establishes a
dual adjunction between $\mbal$ and the category $\KHK$ of what we call compact Hausdorff frames; that is, Kripke frames equipped with a compact Hausdorff topology such that the binary relation is continuous. This dual adjunction restricts to a dual equivalence between $\KHK$ and the reflective subcategory $\mubal$ of $\mbal$ consisting of uniformly complete objects of $\mbal$. This generalizes both Gelfand duality and the duality for modal algebras.
\end{abstract}

\maketitle

\section{Introduction} \label{sec: introduction}

In modal logic there is a well-established duality theory between categories of Kripke frames and the corresponding categories of boolean algebras with operators, which forms the backbone of modern studies of modal logic. One of the most fundamental such dualities establishes that the category of modal algebras is dually equivalent to the category of descriptive frames. This duality originates in the works of J\'onsson and Tarski \cite{JT51},
Halmos \cite{Hal56}, and Kripke \cite{Kri63}. In its current form it was developed by Esakia \cite{Esa74} and Goldblatt \cite{Gol76}.
For a modern account we refer to \cite{SV88} or the textbooks \cite{CZ97,Kra99,BRV01}.

This duality generalizes the celebrated Stone duality between the categories of boolean algebras and Stone spaces (zero-dimensional compact Hausdorff spaces). Descriptive frames are Stone spaces equipped with a continuous relation.
It is well known that a binary relation $R$ on a Stone space $X$ is continuous iff the corresponding map from $X$ to the Vietoris space $\mathcal{V}X$, given by sending each $x\in X$ to its $R$-image, is a well-defined continuous map
(see \cite[Sec.~1]{Esa74} or \cite[Sec.~3]{KKV04}).
Since the Vietoris space $\mathcal{V}X$ of a compact Hausdorff space $X$ is compact Hausdorff, the above consideration allows us to generalize the notion of a descriptive frame to what we call a {\em compact Hausdorff frame}; that is, a compact Hausdorff space equipped with a continuous relation. The category $\KHK$ of compact Hausdorff frames was studied in \cite{BBH15} where Isbell \cite{Isb72} and de Vries \cite{deV62} dualities for the category $\KHaus$ of compact Hausdorff spaces were generalized to $\KHK$.

One of the best known (and oldest) dualities for $\KHaus$ is Gelfand duality, which establishes that $\KHaus$ is dually equivalent to the category $\ubal$ of uniformly complete bounded archimedean $\ell$-algebras (see Section~\ref{sec: Gelfand duality} for details). This duality is obtained by associating to each compact Hausdorff space $X$ the ring $C(X)$ of continuous real-valued functions on $X$.
For some time now there has been a desire to generalize Gelfand duality to a duality for $\KHK$, but it remained elusive for at least two reasons. On the conceptual side, there was no agreement on what should be the definition of modal operators on the ring $C(X)$. On the technical side, it was unclear how to axiomatize attempted definitions of modal operators.

The goal of this paper is to resolve these issues. After recalling Gelfand duality, we define a modal operator on the ring $C(X)$ for each compact Hausdorff frame $(X,R)$, and study its basic properties. This motivates the definition of a modal operator on an arbitrary bounded archimedean $\ell$-algebra, which is the main definition of the paper, giving rise to the category $\mbal$ of modal bounded archimedean $\ell$-algebras. We show that there is a contravariant functor $(-)^*$ from $\KHK$ to $\mbal$.

Next we define a contravariant functor $(-)_*:\mbal\to\KHK$ in the opposite
direction. Proving that $(-)_*:\mbal\to\KHK$ is well defined is technically the most challenging part of the paper. Our main result establishes that the contravariant functors $(-)^*$ and $(-)_*$ yield a dual adjunction between $\mbal$ and $\KHK$, which restricts to a dual equivalence between $\KHK$ and the reflective subcategory $\mubal$ of $\mbal$ consisting of uniformly complete objects of $\mbal$.

Our result generalizes both Gelfand duality
and the duality between modal algebras and descriptive frames. We also take first steps in developing correspondence theory for $\mbal$ by characterizing the classes of algebras in $\mbal$ such that the corresponding relations on the dual side are serial, reflexive, transitive, or symmetric. We conclude the paper outlining several possible future directions of this line of research.

\section{Gelfand duality}\label{sec: Gelfand duality}

Gelfand duality has a long history. In \cite{GN43}, by working with continuous complex-valued functions, Gelfand and Naimark established that $\KHaus$ is dually equivalent to the category of commutative $C^*$-algebras. Independently, Stone \cite{Sto40} worked with continuous real-valued functions and established that $\KHaus$ is dually equivalent to the category of uniformly complete bounded archimedean $\ell$-algebras. The two dualities are closely related as the two categories of algebras are equivalent, which can be seen directly without passing to $\KHaus$. Indeed, the self-adjoint elements of a commutative $C^*$-algebra form an algebra that Stone worked with, and each such algebra $A$ gives rise to a commutative $C^*$-algebra by taking the complexification $A\otimes_\mathbb R\mathbb C$ (see
\cite[Sec.~7]{BMO13a} for details). Because of this, these two dualities are sometimes called by the unifying name of Gelfand-Naimark-Stone duality. We follow \cite[Sec.~IV.4]{Joh82} in calling this Gelfand duality, although our approach is more closely related to Stone's.

We start by recalling several basic definitions (see \cite[Ch.~XIII and onwards]{Bir79} or \cite{BMO13a}). All rings that we will consider in this paper are commutative and unital (have multiplicative identity $1$).

\begin{definition}
\begin{enumerate}
\item[]
\item A ring $A$ with a partial order $\le$ is an \emph{$\ell$-ring} (that is, a \emph{lattice-ordered ring}) if $(A,\le)$ is a lattice, $a\le b$ implies $a+c \le b+c$ for each $c$, and $0 \leq a, b$ implies $0 \le ab$.
\item An $\ell$-ring $A$ is \emph{bounded} if for each $a \in A$ there is $n \in \mathbb{N}$ such that $a \le n\cdot 1$ (that is, $1$ is
a \emph{strong order unit}).
\item An $\ell$-ring $A$ is \emph{archimedean} if for each $a,b \in A$, whenever $n\cdot a \le b$ for each $n \in \mathbb{N}$, then $a \le 0$.
\item An $\ell$-ring $A$ is an \emph{$\ell$-algebra} if it is an $\mathbb R$-algebra and for each $0 \le a\in A$ and $0\le \lambda\in\mathbb R$ we
have $0 \le \lambda \cdot a$.
\item Let $\bal$ be the category of bounded archimedean $\ell$-algebras and unital $\ell$-algebra homomorphisms.
\end{enumerate}
\end{definition}

Let $A\in\bal$. For $a\in A$, define the \emph{absolute value} of $a$ by
\[
|a|=a\vee(-a)
\]
and the \emph{norm} of $a$ by
\[
||a||=\inf\{\lambda\in\mathbb R \mid |a|\le \lambda\}.\footnote{We view $\mathbb{R}$ as an $\ell$-subalgebra of $A$ by identifying
$\lambda\in\mathbb R$ with $\lambda\cdot 1\in A$.}
\]
Then $A$ is \emph{uniformly complete} if the norm is complete. Let $\ubal$ be the full subcategory of $\bal$ consisting of uniformly complete $\ell$-algebras.

\begin{theorem} [Gelfand duality \cite{GN43,Sto40}]
There is a dual adjunction between $\bal$ and $\KHaus$ which restricts to a dual equivalence between $\KHaus$ and $\ubal$.
\[
\begin{tikzcd}
\ubal \arrow[rr, hookrightarrow] && \bal \arrow[dl, "(-)_*"]  \arrow[ll, bend right = 20] \\
&  \KHaus \arrow[ul,  "(-)^*"] &
\end{tikzcd}
\]
\end{theorem}

The functors $(-)^*:\KHaus \to \bal$ and $(-)_*:\bal\to\KHaus$ establishing the dual adjunction are defined as follows. For a compact Hausdorff space $X$ let $X^*$ be the ring $C(X)$ of (necessarily bounded) continuous real-valued functions on $X$. For a continuous map $\varphi:X\to Y$ let $\varphi^*:C(Y)\to C(X)$ be defined by $\varphi^*(f)=f\circ\varphi$ for each $f\in C(Y)$. Then $(-)^*:\KHaus\to\bal$ is a well-defined contravariant functor.

For $A\in\bal$, we recall that an ideal $I$ of $A$ is an \emph{$\ell$-ideal} if $|a|\le|b|$ and $b\in I$ imply $a\in I$, and that $\ell$-ideals are exactly the kernels of $\ell$-algebra homomorphisms. Let $Y_A$ be the space of maximal $\ell$-ideals of $A$, whose closed sets are exactly sets of the form
\[
Z_\ell(I) = \{M\in Y_A\mid I\subseteq M\},
\]
where $I$ is an $\ell$-ideal of $A$. The space $Y_A$ is often referred to as the \emph{Yosida space} of $A$, and it is well known that $Y_A\in\KHaus$. We then set $A_*=Y_A$. For a morphism $\alpha$ in $\bal$ let $\alpha_*=\alpha^{-1}$. Then $(-)_*:\bal\to\KHaus$ is a well-defined contravariant functor, and the functors $(-)_*$ and $(-)^*$ yield a dual adjunction between $\bal$ and $\KHaus$.

Moreover, for $X\in\KHaus$ we have that $\varepsilon_X:X\to (X^*)_*$
is a homeomorphism where
\[
\varepsilon_X(x)=\{f\in C(X) \mid f(x)=0\}.
\]
Furthermore, for $A\in\bal$ define $\zeta_A :A\to (A_*)^*$ by $\zeta_A(a)(M)=\lambda$ where $\lambda$ is the unique real number satisfying $a+M=\lambda+M$. Then $\zeta_A$ is a monomorphism in $\bal$ separating points of $Y_A$. Therefore, by the Stone-Weierstrass theorem, we have:

\begin{proposition}\label{prop: SW}
\begin{enumerate}
\item[]
\item The uniform completion of $A \in \bal$ is $\zeta_A : A \to C(Y_A)$. Therefore, if $A$ is uniformly complete, then $\zeta_A$ is an isomorphism.
\item $\ubal$ is a reflective subcategory of $\bal$, and the reflector $\zeta : \bal \to \ubal$ assigns to each $A \in \bal$ its uniform completion $C(Y_A) \in \ubal$.
\end{enumerate}
\end{proposition}

Consequently, the dual adjunction restricts to a dual equivalence between $\ubal$ and $\KHaus$, yielding Gelfand duality.

\section{Modal operators on $C(X)$} \label{sec: from frames to rings}

In this section we define modal operators on rings of continuous real-valued functions on compact Hausdorff frames and study their basic properties.
This motivates the definition of a modal operator on $A\in\bal$, giving rise to the category $\mbal$ of modal bounded archimedean $\ell$-algebras. We end the section by describing a
contravariant functor from $\KHK$ to $\mbal$.

We recall that a \emph{Kripke frame} is a pair $\mathfrak F=(X,R)$ where $X$ is a set and $R$ is a binary relation on $X$. As usual, for $x\in X$
we write
\[
R[x]=\{y\in X \mid xRy\} \quad\mbox{and} \quad R^{-1}[x]=\{y\in X \mid yRx\},
\]
and for $U\subseteq X$ we write
\[
R[U]=\bigcup\{R[u] \mid u\in U\} \quad\mbox{and}\quad R^{-1}[U]=\bigcup\{R^{-1}[u] \mid u\in U\}.
\]

\begin{definition} \cite{BBH15}
A binary relation $R$ on a compact Hausdorff space $X$ is \emph{continuous} if:
\begin{enumerate}
\item $R[x]$ is closed for each $x\in X$.
\item $F\subseteq X$ closed implies $R^{-1}[F]$ is closed.
\item $U\subseteq X$ open implies $R^{-1}[U]$ is open.
\end{enumerate}
 If $R$ is a continuous relation on $X$, we call $(X, R)$ a \emph{compact Hausdorff frame}.
\label{def:cont for KHaus}
\end{definition}

\begin{notation} \label{not: D and E}
For a binary relation $R$ on a set $X$ let
\begin{align*}
D &= \{ x \in X \mid R[x] \ne \varnothing \} = R^{-1}[X],\\
E &= X \setminus D = \{ x \in X \mid R[x] = \varnothing\}.
\end{align*}
\end{notation}

The next lemma is straightforward and we omit the proof.

\begin{lemma} \label{lem: clopen}
If $(X,R)$ is a compact Hausdorff frame, then $D$ and $E$ are clopen subsets of $X$.
\end{lemma}

\begin{definition}\label{def: Box_R}
For a compact Hausdorff frame $(X,R)$, define $\Box_R$ on $C(X)$ by
\[
(\Box_Rf)(x) = \left\{\begin{array}{ll}\inf fR[x] & \textrm{if }x \in D\\ 1 & \textrm{otherwise.}\end{array}\right.
\]
\end{definition}

\begin{remark} \label{rem: sup}
We define $\Diamond_R$ by
\[
(\Diamond_R f)(x) = \left\{\begin{array}{ll}\sup fR[x] & \textrm{if }x \in D\\ 0 & \textrm{otherwise.}\end{array}\right.
\]
We have
\[
\Diamond_R f = 1 -\Box_R (1-f) \quad\mbox{and}\quad \Box_Rf = 1 -\Diamond_R(1-f).
\]
For, if $x \in D$, then
\begin{align*}
1 - \Box_R(1-f)(x) &= 1 -\inf\{ 1-f(y) \mid xRy\} = 1 -(1 - \sup\{ f(y)  \mid xRy\}) \\
&= \sup \{ f(y) \mid xRy\} = \Diamond_Rf(x).
\end{align*}
If $x \in E$, then $(1 - \Box_R (1-f))(x) = 1 - 1 = 0 = (\Diamond_Rf)(x)$. Thus, $\Diamond_R f = 1 - \Box_R(1-f)$, as desired.
A similar argument yields $\Box_Rf = 1 - \Diamond_R(1-f)$. Therefore, each of $\Box_R$ and $\Diamond_R$ can be determined from the other.
\end{remark}

\begin{remark}
Let $(X,R)$ be a compact Hausdorff frame, $f \in C(X)$, and $x \in X$ with $R[x] \ne \varnothing$. Then $fR[x]$ is a nonempty compact subset of $\mathbb{R}$, and so it has least and greatest elements. Thus, we have
\[
(\Box_Rf)(x) = \min fR[x] \quad\mbox{and}\quad (\Diamond_R f)(x) = \max fR[x].
\]
\end{remark}

\begin{lemma} \label{lem: Box on C}
Let $(X,R)$ be a compact Hausdorff frame. If $f \in C(X)$, then $\Box_R f \in C(X)$.
\end{lemma}

\begin{proof}
To see that $\Box_R f$ is continuous, it is sufficient to show that for each $\lambda \in \mathbb{R}$, both $(\Box_R f)^{-1}(\lambda, \infty)$ and
$(\Box_R f)^{-1}(-\infty, \lambda)$ are open in $X$. We first show that $(\Box_R f)^{-1}(\lambda, \infty)$ is open.
Let $x \in X$ and first suppose that $x \in D$. Then $fR[x]$ is a nonempty compact subset of $\mathbb{R}$, so it has a least element. Therefore,
\begin{eqnarray*}
x \in (\Box_R f)^{-1}(\lambda, \infty) & \mbox{ iff } & (\Box_R f)(x) > \lambda \\
& \mbox{ iff } & \min(fR[x]) > \lambda \\
& \mbox{ iff } & R[x] \subseteq f^{-1}(\lambda, \infty) \\
& \mbox{ iff } & x \in  X \setminus R^{-1}[X \setminus f^{-1}(\lambda, \infty)].
\end{eqnarray*}
Next suppose that $x \in E$. Then $(\Box_R f)(x) = 1$. Thus, $E \subseteq (\Box_R f)^{-1}(\lambda, \infty)$ if $\lambda < 1$, and
$E \cap (\Box_R f)^{-1}(\lambda, \infty) = \varnothing$ otherwise. Consequently,
\[
(\Box_R f)^{-1}(\lambda, \infty) = \left[D \cap (X \setminus R^{-1}[X \setminus f^{-1}(\lambda, \infty)]) \right]\cup E
\]
if $\lambda < 1$, and
\[
(\Box_R f)^{-1}(\lambda, \infty) = D \cap (X \setminus R^{-1}[X \setminus f^{-1}(\lambda, \infty)])
\]
if $1 \le \lambda$. Since $f\in C(X)$ and
$R$ is continuous, $X \setminus R^{-1}[X \setminus f^{-1}(\lambda, \infty)]$ is open. Thus, $(\Box_R f)^{-1}(\lambda, \infty)$ is open
by Lemma~\ref{lem: clopen}.

We next show that $(\Box_R f)^{-1}(-\infty, \lambda)$ is open.
If $x \in D$, then
\begin{eqnarray*}
x \in (\Box_R f)^{-1}(-\infty, \lambda) & \mbox{  iff } & (\Box_R f)(x) < \lambda \\
& \mbox{ iff } & \min(fR[x]) < \lambda \\
& \mbox{ iff } & R[x] \cap f^{-1}(-\infty, \lambda) \ne \varnothing \\
& \mbox{ iff } & x \in R^{-1}[f^{-1}(-\infty, \lambda)].
\end{eqnarray*}
If $\lambda \le 1$, then $E \cap (\Box_R f)^{-1}(-\infty, \lambda) = \varnothing$, and if $1 < \lambda$, then
$E \subseteq (\Box_R f)^{-1}(-\infty, \lambda)$. Therefore,
\[
(\Box_R f)^{-1}(-\infty, \lambda) = D \cap R^{-1}[f^{-1}(-\infty, \lambda)]
\]
if $\lambda \le 1$,
and
\[
(\Box_R f)^{-1}(-\infty, \lambda) =\left[D \cap (R^{-1}[f^{-1}(-\infty, \lambda)])\right] \cup E
\]
if $\lambda > 1$. Since $f\in C(X)$ and $R$ is continuous,
$R^{-1}[f^{-1}(-\infty, \lambda)]$ is open. Consequently, $(\Box_R f)^{-1}(-\infty, \lambda)$ is open
by Lemma~\ref{lem: clopen}.
This completes the proof that if $f \in C(X)$, then $\Box_R f \in C(X)$.
\end{proof}

In the next lemma we describe the properties of $\Box_R$. For this we recall (see, e.g., \cite[Rem~2.2]{BMO13a}) that if $A\in\bal$ and $a\in A$, then the \emph{positive} and \emph{negative} parts of $a$ are defined as
\[
a^+=a\vee 0 \quad \mbox{and}\quad a^-=-(a\wedge 0)=(-a)\vee 0.
\]
Then $a^+,a^- \ge 0$, $a^+ \wedge a^-=0$, $a=a^+ - a^-$, and $|a|=a^+ + a^-$.

\begin{lemma} \label{lem: properties of box}
Let $(X,R)$ be a compact Hausdorff frame, $f,g \in C(X)$, and $\lambda \in \mathbb{R}$.
\begin{enumerate}
\item $\Box_R(f\wedge g) = \Box_R f \wedge \Box_R g$. In particular, $\Box_R$ is order preserving.
\item $\Box_R \lambda = \lambda + (1-\lambda)(\Box_R 0)$. In particular, $\Box_R 1 = 1$.
\item $\Box_R (f^+) = (\Box_R f)^+$.
\item $\Box_R(f + \lambda) = \Box_R f + \Box_R\lambda - \Box_R 0$.
\item If $0 \le \lambda$, then $\Box_R (\lambda f) = (\Box_R\lambda)( \Box_R f)$.
\end{enumerate}
\end{lemma}

\begin{proof}
(1). For $x \in D$, we have
\begin{align*}
\Box_R (f \wedge g )(x) &= \inf \{ (f \wedge g)(y) \mid  y \in R[x] \} = \inf \{ \min \{ f(y), g(y) \} \mid y \in R[x] \} \\
&= \min \{  \inf \{ f(y) \mid y \in R[x] \},  \inf \{ g(y) \mid y \in R[x] \} \}\\
& = \min \{ (\Box_R f)(x), (\Box_R g)(x)  \} \\
&= (\Box_R f \wedge \Box_R g)(x).
\end{align*}
If $x \in E$, then $\Box_R (f \wedge g)(x) = 1 = (\Box_R f \wedge \Box_R g)(x)$.
Thus, $\Box_R(f\wedge g) = \Box_R f \wedge \Box_R g$.

(2). For $x \in D$, if $\mu \in \mathbb{R}$, we have $(\Box_R \mu)(x) = \inf \{ \mu \mid y \in R[x] \} = \mu$. From this we see that
$(\Box_R \lambda)(x) = \lambda =  (\lambda + (1-\lambda)(\Box_R 0))(x)$. If $x \in E$, then
$(\Box_R \lambda)(x) = 1 = (\lambda + (1-\lambda)(\Box_R 0))(x)$. Thus, $\Box_R \lambda = \lambda =  \lambda + (1-\lambda)(\Box_R 0)$.
Setting $\lambda=1$ yields $\Box_R 1 = 1$.

(3). For $x \in D$, we have
\begin{align*}
(\Box_R (f^+))(x) &= \Box_R (f \join 0)(x) =\inf \{ \max\{f(y),0 \} \mid y \in R[x] \} \\
&=\max\{\inf \{ f(y) \mid y \in R[x] \},0\} = \max \{ \Box_R f(x), 0\}\\
&= (\Box_R f \vee 0)(x) = (\Box_R f)^+(x).
\end{align*}
If $x \in E$, then $(\Box_R (f^+))(x) = 1 = (\Box_R f)^+(x)$.
Thus, $\Box_R (f^+) = (\Box_R f)^+$.

(4). For $x \in D$, we have
\begin{align*}
\Box_R (f+\lambda)(x) &=\inf \{ f(y)+\lambda \mid y \in R[x] \}\\
&=\inf \{ f(y) \mid y \in R[x] \}+\lambda  \\
&= \Box_R f(x) + \lambda.
\end{align*}
On the other hand,
\[
(\Box_R f + \Box_R \lambda - \Box_R 0)(x) = (\Box_R f)(x) + (\Box_R \lambda)(x) - (\Box_R 0)(x) = (\Box_R f)(x) + \lambda.
\]
Therefore, $\Box_R (f+\lambda)(x) = (\Box_R f + \Box_R \lambda - \Box_R 0)(x)$.
If $x \in E$, then $\Box_R (f+\lambda)(x)  = 1 = (\Box_Rf  + \Box_R \lambda - \Box_R 0)(x)$.
Thus, $\Box_R(f + \lambda) = \Box_Rf  + \Box_R \lambda - \Box_R 0$.

(5). Let $0 \le \lambda$. For $x \in D$, we have
\begin{align*}
(\Box_R \lambda f )(x) &=\inf \{ \lambda f(y) \mid y \in R[x] \}= \lambda \inf \{ f(y) \mid y \in R[x] \} \\
&= \lambda (\Box_R f)(x) = (\Box_R\lambda)(x)(\Box_R f)(x) = (\Box_R\lambda \Box_R f)(x).
\end{align*}
If $x \in E$, then $(\Box_R \lambda f)(x) = 1 = (\Box_R \lambda)(\Box_R f)(x)$.
Thus, $\Box_R (\lambda f) = (\Box_R\lambda)(\Box_R f)$.
\end{proof}

\begin{remark} \label{rem: properties of diamond}
Lemma~\ref{lem: properties of box} can be stated dually in terms of $\Diamond_R$ as follows. Let $(X,R)$ be a compact Hausdorff frame, $f,g \in C(X)$, and $\lambda \in \mathbb{R}$.
\begin{enumerate}
\item $\Diamond_R (f \vee g)=\Diamond_R f \vee \Diamond_R g$. In particular, $\Diamond_R$ is order preserving.
\item $\Diamond_R \lambda = \lambda(\Diamond_R 1)$. In particular, $\Diamond_R 0 = 0$.
\item $\Diamond_R (f \wedge 1) = (\Diamond_R f) \wedge 1$.
\item $\Diamond_R(f+\lambda)=\Diamond_R f + \Diamond_R\lambda$.
\item If $0 \le \lambda$, then $\Diamond_R (\lambda f) = \Diamond_R\lambda \Diamond_R f$.
\end{enumerate}

The identities (1), (3),
and (5) are direct translations of the corresponding identities for $\Box_R$. However, the identities (2) and (4) are
simpler. We next show why $\Diamond_R$ affords such simplifications.

For (2), since $\Diamond_R 1 = 1 - \Box_R 0$, by Lemma~\ref{lem: properties of box}(2),
\[
\Diamond_R \lambda = 1 - \Box_R(1-\lambda) = 1 - (1 - \lambda + \lambda\Box_R 0) = \lambda(1 - \Box_R 0) = \lambda\Diamond_R 1.
\]

For (4), by using (4) and (2) of Lemma~\ref{lem: properties of box}, we have
\begin{align*}
\Diamond_R(f + \lambda) &= 1 - \Box_R(1- (f + \lambda)) = 1 - \Box_R((1-f) - \lambda) \\
&= 1 - (\Box_R(1-f) + \Box_R(-\lambda) - \Box_R 0) = \Diamond_R f - \Box_R(-\lambda) + \Box_R 0 \\
&= \Diamond_R f - (-\lambda + (1+\lambda)\Box_R 0) + \Box_R 0 = \Diamond_R f + \lambda(1 - \Box_R 0) = \Diamond_R f + \Diamond_R \lambda.
\end{align*}

In Remark~\ref{rem: why box} we explain why we prefer to work with $\Box_R$.
\end{remark}

Lemmas~\ref{lem: Box on C} and~\ref{lem: properties of box} motivate the main definition of this paper.

\begin{definition}\label{def:mbal}
\begin{enumerate}
\item[]
\item Let $A \in \bal$. We say that a unary function $\Box: A \to A$ is a \textit{modal operator} on $A$ provided $\Box$ satisfies the following
axioms for each $a,b \in A$ and $\lambda \in \mathbb{R}$:
\begin{enumerate}
\item[(M1)] $\Box (a \meet b) = \Box a \meet \Box b$.
\item[(M2)] $\Box \lambda = \lambda + (1 - \lambda)\Box 0$.
\item[(M3)] $\Box (a^+)=(\Box a)^+$.
\item[(M4)] $\Box (a+\lambda)=\Box a + \Box\lambda - \Box 0$.
\item[(M5)] $\Box (\lambda a)= (\Box\lambda)(\Box a)$ provided $\lambda\ge 0$.
\end{enumerate}
\item If $\Box$ is a modal operator on $A\in\bal$, then we call the pair $\FA = (A, \Box)$ a \textit{modal bounded archimedean $\ell$-algebra}.
\item Let $\mbal$ be the category of modal bounded archimedean $\ell$-algebras and unital $\ell$-algebra homomorphisms preserving $\Box$.
\end{enumerate}
\end{definition}

\begin{remark}
We can define $\Diamond:A \to A$ dual to $\Box$ by $\Diamond a = 1 - \Box (1 - a)$ for each $a \in A$. Then $(A,\Diamond)$ satisfies the axioms for $\Diamond$ dual to the ones for $\Box$ in Definition~\ref{def:mbal}(1) (see Remark~\ref{rem: properties of diamond}). While algebras in $\mbal$ can be axiomatized either in the signature of $\Box$ or $\Diamond$, we prefer to work with $\Box$ for the reasons given in Remark~\ref{rem: why box}.
\end{remark}

\begin{remark} \label{rem: serial}
If $\Box 0 = 0$, then (M2), (M4), and (M5) simplify to the following:
\begin{enumerate}
\item[(M2${}'$)]$\Box \lambda = \lambda$.
\item[(M4${}'$)]$\Box (a + \lambda) = \Box a + \lambda$.
\item[(M5${}'$)]$\Box (\lambda a) = \lambda \Box a$ provided $\lambda \ge 0$.
\end{enumerate}
Moreover, (M2${}'$) follows from (M4${}'$) by setting $a = 0$. Furthermore, $\Diamond a = -\Box(-a)$.
In Section~\ref{subsec: correspondence} we will see that $\Box 0 = 0$ holds iff the binary relation $R_\Box$ on the Yosida space of $A$ is serial (see Definition~\ref{def: R_box on Y_A}).
\end{remark}

\begin{lemma}\label{lem: properties}
Let $(A, \Box) \in\mbal$, $a,b\in A$, and $\lambda\in\mathbb R$.
\begin{enumerate}
\item $a \leq b$ implies $\Box a \leq \Box b$.
\item $\Box 1 = 1$.
\item $a \ge 0$ implies $\Box a \ge 0$.
\item $(\Box 0)(\Box a) = \Box0$. In particular, $\Box 0$ is an idempotent.
\item $\Box(a + \lambda) = \Box a + \lambda(1-\Box 0)$.
\item $\Diamond a = -\Box(-a)(1 - \Box 0)$.
\item $(\Diamond a)(\Box 0) = 0$.
\end{enumerate}
\end{lemma}

\begin{proof}
(1). If $a \leq b$, then $a \meet b=a$. Therefore, by (M1), $\Box a = \Box (a \meet b)= \Box a \meet \Box b$. Thus, $\Box a \leq \Box b$.

(2). This follows by substituting $\lambda=1$ in (M2).

(3). From (M3) and $a \ge 0$ we have $\Box a = \Box (a^+) = (\Box a)^+ \ge 0$.

(4). By (M5), $\Box 0 = \Box (0a) = (\Box 0)(\Box a)$. Setting $a = 0$ gives $(\Box 0)^2 = \Box 0$.

(5). By (M4), $\Box(a + \lambda) = \Box a + \Box \lambda - \Box 0$. By (M2),
$\Box \lambda = \lambda + (1-\lambda)(\Box 0) = \lambda(1 - \Box 0) + \Box 0$. Therefore, $\Box \lambda - \Box 0 = \lambda(1 - \Box 0)$,
and so (5) follows.

(6). By (M4), (2), and (4) we have
\begin{align*}
\Diamond a &= 1 - \Box(1-a) = 1 - (\Box(-a) + \Box 1 - \Box 0) \\
&= -\Box(-a) + \Box 0 = -\Box(-a) + \Box(-a)\Box 0 \\
&= -\Box(-a)(1- \Box 0).
\end{align*}

(7). Since $\Box 0$ is an idempotent by (4), we have $(1 - \Box 0)\Box 0 = 0$. Multiplying both sides of (6) by $\Box 0$ yields $\Diamond a \Box 0 = 0$.
\end{proof}

As follows from Lemmas~\ref{lem: Box on C} and~\ref{lem: properties of box}, if $(X,R)$ is a compact Hausdorff frame, then $(C(X),\Box_R)\in\mbal$. We now extend this correspondence to a contravariant functor. For this we recall the definition of a bounded morphism.

\begin{definition}
\begin{enumerate}
\item[]
\item A \emph{bounded morphism} (or \emph{p-morphism}) between Kripke frames $\mathfrak F=(X,R)$ and $\mathfrak G=(Y,S)$ is a map $f:X\to Y$ satisfying $f(R[x])=S[f(x)]$ for each  $x\in X$ (equivalently, $f^{-1} (S^{-1}[y]) = R^{-1}[f^{-1}(y)]$ for each $y \in Y$).
\item Let $\KHK$ be the category of compact Hausdorff frames and continuous bounded morphisms.
\end{enumerate}
\end{definition}

\begin{lemma} \label{lem: varphi* a morphism}
If $\mathfrak F=(X,R)$ and $\mathfrak G=(Y,S)$ are compact Hausdorff frames and $\varphi : X \to Y$ is a continuous bounded morphism, then $\varphi^*$ is a morphism in $\mbal$.
\end{lemma}

\begin{proof}
That $\varphi^*$ is a $\bal$-morphism follows from Gelfand duality. Therefore, it is sufficient to prove that $\varphi^*$ preserves $\Box$; that is, $\varphi^*(\Box_S f) = \Box_R \varphi^*(f)$ for each $f\in C(Y)$.
Since $\varphi$ is a bounded morphism, $\varphi(R[x]) = S[\varphi(x)]$ for each $x \in X$. Let $x \in X$ and $f\in C(Y)$. If
$R[x] \ne \varnothing$, then $S[\varphi(x)] \ne \varnothing$, so
\begin{equation*}
\begin{split}
\varphi^*(\Box_S f)(x) &= (\Box_S f \circ \varphi)(x)=(\Box_S f)(\varphi(x)) =\inf (f(S[\varphi(x)]))\\
&=\inf (f(\varphi(R[x])))= \inf ((f \circ \varphi)(R[x])) =\Box_R(f \circ \varphi)(x) \\
&= \Box_R(\varphi^*(f))(x).
\end{split}
\end{equation*}
If $R[x] = \varnothing$, then $S[\varphi(x)] = \varnothing$, so
$\varphi^*(\Box_S f)(x) = (\Box_S f)(\varphi(x)) = 1 = (\Box_R \varphi^*(f))(x)$.
Thus, $\varphi^*(\Box_S f) = \Box_R \varphi^*(f)$.
\end{proof}

\begin{theorem}
There is a contravariant functor $(-)^*:\KHK\to\mbal$ which sends $\mathfrak F=(X,R)$ to $\mathfrak F^*=(C(X),\Box_R)$ and a morphism $\varphi$ in $\K$ to $\varphi^*$.
\end{theorem}

\begin{proof}
If $\mathfrak F \in \KHK$, then $\mathfrak F^* \in \mbal$ by Lemmas~\ref{lem: Box on C} and~\ref{lem: properties of box}. If $\varphi$ is a morphism in $\KHK$, then $\varphi^*$ is a morphism in $\mbal$ by Lemma~\ref{lem: varphi* a morphism}. It is elementary to see that $(\psi\circ\varphi)^*=\varphi^*\circ\psi^*$ and that $(-)^*$ preserves identity morphisms. Thus, $(-)^*$ is a contravariant functor.
\end{proof}

\section{Continuous relations on the Yosida space} \label{sec: from rings to frames}

In this section we define a contravariant functor $(-)_*:\mbal\to\KHK$ in the other direction, which is technically the most involved part of the paper.

Let $A\in\bal$. For $S \subseteq A$ let
\[
S^+=\{a \in S \mid a \geq 0 \}.
\]
We point out that if $I$ is an $\ell$-ideal of $A$, then
$I^+ = \{ a^+ \mid a \in I\}$.

\begin{definition} \label{def: R_box on Y_A}
Let $(A,\Box)\in\mbal$ and let $Y_A$ be the Yosida space of $A$. Define $R_\Box$ on $Y_A$ by
\begin{equation*}
x R_\Box y \quad\mbox{iff}\quad \Box y^+ \subseteq x, \quad\mbox{iff}\quad y^+ \subseteq \Box^{-1}x.
\end{equation*}
\end{definition}

\begin{remark} \label{rem: why box}
We have that $xR_\Box y$ iff $(\forall a \ge 0)(a+y = 0+y \Rightarrow \Box a + x = 0 + x)$. If we work with $\Diamond$ instead of $\Box$, since $\Diamond a=1-\Box(1-a)$, the definition becomes $x R_\Box y$ iff $(\forall b \le 1)(b + y = 1 + y \Rightarrow \Diamond b + x = 1 + x)$. Thus, $x R_\Box y$ iff $\{ 1 - \Diamond b \mid 1-b \in y, b \le 1\} \subseteq x$. This more complicated definition is one reason why we prefer to work with $\Box$ rather than $\Diamond$. Another is that, as is standard in working with ordered algebras, using $\Box$ allows us to work with the positive cone rather than the set of elements below 1.
\end{remark}

Let $A \in \bal$. We recall that the \emph{zero set} of $a \in A$ is defined as
\[
\zer(a) = \{ x \in Y_A \mid a \in x \}.
\]
If $S \subseteq A$, then we set
\[
\zer(S) = \bigcap \{ \zer(a) \mid a \in S\} = \{ x \in Y_A \mid S \subseteq x\}.
\]
It is easy to see that if $I$ is the $\ell$-ideal generated by $S$, then $\zer(S) = \zer(I)$. We define the \textit{cozero set of $S$} as
\[
\cozer(S)=Y_A \setminus \zer(S)=\{ x \in Y_A \mid S \not\subseteq x \}.
\]
Since the zero sets are exactly the closed sets, the cozero sets are exactly the open sets of $Y_A$.
The family $\{ \cozer(a) \mid a \in A \}$ then constitutes a basis for the topology on $Y_A$.

\begin{remark} \label{rem: properties of primes}
Let $A \in \bal$, $Y_A$ be the Yosida space of $A$, $x \in Y_A$, and $a \in A$.
\begin{enumerate}
\item $x$ is a prime ideal of $A$ because $A/x \cong \mathbb{R}$ (see, e.g., \cite[Cor.~2.7]{HJ61}).
\item Either $a^+ \in x$ or $a^- \in x$. This follows from (1) and $a^+ a^- = 0$.
\item $a^+ \in x$ and $a^- \notin x$ iff $a + x < 0 + x$ (see \cite[Rem.~2.11]{BMO16}).
\item $a^+ \in x$ iff $a + x \le 0 + x$. For, if $a^+ \in x$, then $a + x = (a^+ - a^-) + x = -a^- + x \le 0 + x$ since $a^- \ge 0$.
Conversely, if $a + x \le 0 + x$, then either $a + x < 0 + x$, in which case $a^+ \in x$ by (3), or $a + x = 0 + x$, in which case
$a \in x$, so $a^+ \in x$.
\item $a^- \in x$ and $a^+ \notin x$ iff $a + x > 0 + x$ (see \cite[Rem.~2.11]{BMO16}).
\item $a^- \in x$ iff $a + x \ge 0 + x$. The proof is similar to that of (4) but uses (5) instead of (3).
\end{enumerate}
\end{remark}

Recalling Notation~\ref{not: D and E}, if $(Y_A,R_\Box)$ is the dual of
$(A, \Box) \in \mbal$,
then we denote $R_\Box^{-1}[Y_A]$ by $D_A$ and $Y_A \setminus D_A$ by $E_A$.

\begin{lemma} \label{lem: cosets}
Let $(A, \Box) \in \mbal$, $a \in A$, $\lambda \in \mathbb{R}$, and $x \in Y_A$.
\begin{enumerate}
\item If $x \in D_A$, then $\Box 0 \in x$.
\item If $\Box 0 \in x$, then $\Box(a + \lambda) + x = (\Box a + \lambda) + x$.
\item If $\Box 0 \in x$, then $\Box((a - \lambda)^+) \in x$ iff $(\Box a - \lambda)^+ \in x$.
\item If $\Box 0 \in x$, then $\Diamond a + x = -\Box(-a) + x$.
\item If $\Box 0 \notin x$, then $1 - \Box a  \in x$.
\item If $\Diamond a \notin x$, then $\Box 0 \in x$.
\end{enumerate}
\end{lemma}

\begin{proof}
(1). If $x \in D_A$, then there is $y$ with $x R_\Box y$. Therefore, since $0 \in y^+$, we have $\Box 0 \in x$.

(2). By (M4) and (M2), $\Box(a + \lambda) = \Box a + \lambda - \lambda\Box 0$. Therefore, if $\Box 0 \in x$, then
$\Box(a + \lambda) + x = (\Box a + \lambda) + x$.

(3). This follows from (M3), Remark~\ref{rem: properties of primes}(4), and (2).

(4). Apply Lemma~\ref{lem: properties}(6).

(5). By Lemma~\ref{lem: properties}(4), $\Box 0 = (\Box 0)(\Box a)$, so $(\Box 0)(1 - \Box a) = 0 \in x$. Since $\Box 0 \notin x$ and $x$ is a prime ideal,
$1 - \Box a \in x$.

(6). By Lemma~\ref{lem: properties}(7), $(\Diamond a)(\Box 0) = 0 \in x$. Since $x$ is a prime ideal and $\Diamond a \notin x$, we have $\Box 0 \in x$.
\end{proof}

\begin{proposition}\label{prop: R[x]}
$R_\Box[x]$ is closed for every $x \in Y_A$.
\end{proposition}

\begin{proof}
We prove that $Y_A \setminus R_\Box[x]$ is open for every $x \in Y_A$. Let $y \notin R_\Box[x]$, so
$y^+ \nsubseteq \Box^{-1} x$. Therefore, there is $a \geq 0$ such that $a \in y$ and $\Box a \notin x$. By Lemma~\ref{lem: properties}(3),
$\Box a \geq 0$, so there is $0 \le \lambda \in \mathbb{R}$
such that $(\Box a -\lambda) +x > 0+x$ but $(a-\lambda)+y < 0+y$. By Remark~\ref{rem: properties of primes}(3), $(a-\lambda)^- \notin y$ and
$(\Box a - \lambda)^+ \notin x$. Thus, $y\in\cozer((a-\lambda)^-)$, and
it remains to show that $\cozer((a-\lambda)^-) \cap R_\Box[x]= \varnothing$. Suppose not. Then there is $z$ such that $x R_\Box z$ and
$z \in \cozer((a-\lambda)^-)$. Since $z$ is a prime ideal and $(a-\lambda)^- \notin z$, we have $(a-\lambda)^+ \in z$
(see Remark~\ref{rem: properties of primes}(2)). But $x R_\Box z$ means $z^+ \subseteq \Box^{-1} x$, so $\Box 0, \Box((a-\lambda)^+) \in x$. Thus, by (M3) and Lemma~\ref{lem: cosets}(3),
$(\Box a - \lambda)^+ \in x$, hence $(\Box a - \lambda) + x \le 0 + x$. The obtained contradiction proves that
$\cozer((a-\lambda)^-) \cap R_\Box[x]= \varnothing$, completing the proof.
\end{proof}

For a topological space $X$ and a continuous real-valued function $f$ on $X$, we recall that the \emph{zero set} of $f$ is
\[
Z(f)=\{x\in X \mid f(x)=0\}
\]
and the \emph{cozero set} of $f$ is
\[
\mathrm{coz}(f) = X \setminus Z(f) = \{x\in X \mid f(x) \ne 0\}.
\]
The following lemma is a consequence of
\cite[Prob.~1D, p. 21]{GJ60}.

\begin{lemma} \label{lem: GJ}
Let $A\in\bal$ and $a,s \in A$. If $\zer(a) \subseteq \interi \zer(s)$, then there is $f \in C(Y_A)$
such that $\zeta_A(s)=\zeta_A(a)f$ in $C(Y_A)$.
\end{lemma}

\begin{proof}
Observe that for each $t\in A$ we have $\zer(t)=Z(\zeta_A(t))$. Therefore, $\zer(a) \subseteq \interi \zer(s)$ implies
$Z(\zeta_A(a)) \subseteq \interi Z(\zeta_A(s))$. Now apply \cite[Prob.~1D, p. 21]{GJ60}.
\end{proof}

\begin{lemma} \label{lemmaunion}
Let $(A,\Box)\in\mbal$, $x \in Y_A$, $S=(A \setminus \Box^{-1} x)^+$, and $a \in (\Box^{-1} x)^+$.
\begin{enumerate}
\item $\bigcap\{ \cozer(s) \mid s \in S \} = \bigcap\{ \overline{\cozer(s)} \mid s \in S \}$ for every $s \in S$.
\item $\overline{\cozer(s)} \cap \zer(a) \neq \varnothing$ for every $s \in S$.
\item The family $\{ \overline{\cozer(s)} \cap \zer(a) \mid s \in S \}$ has the finite intersection property.
\end{enumerate}
\end{lemma}

\begin{proof}
(1). The inclusion $\subseteq$ is clear. To prove the reverse inclusion, it is sufficient to prove that for each $s \in S$ there is
$t \in S$ such that $\overline{\cozer(t)} \subseteq \cozer(s)$. Since $s \in S$, there is $\varepsilon \in \mathbb{R}$
with $\Box s + x > \varepsilon + x > 0 + x$. Set $t=(s-\varepsilon)^+$. Then $t \geq 0$ and
\[
\Box t = \Box (s-\varepsilon)^+  = (\Box (s - \varepsilon))^+
\]
by (M3). If $\Box t \in x$, then $\Box (s - \varepsilon) + x \le 0 + x$, so $\Box s - \varepsilon(1-\Box 0) + x \le 0 + x$ by
Lemma~\ref{lem: properties}(5). We have $\Box 0 \in x$ by Lemma~\ref{lem: cosets}(5) as $\Box a \in x$, so $\Box  s - \varepsilon \le 0 + x$,
and hence $\Box s + x \le \varepsilon + x$. The obtained contradiction shows $\Box t \notin x$, so $t \in S$. Let $ z \in \zer(s)$. Then
$z \in \zeta_A(s)^{-1}(- \varepsilon, \varepsilon)$, an open set. But $\zeta_A(s)^{-1}(- \varepsilon, \varepsilon) \subseteq \zer(t)$ by
definition of $t$ and Remark~\ref{rem: properties of primes}(3), so $\zer(s) \subseteq \interi \zer(t)$. Thus,
$\overline{\cozer(t)} \subseteq \cozer(s)$.

(2). Note that $\overline{\cozer(s)} \cap \zer(a) \neq \varnothing$ means that $\zer(a) \nsubseteq \interi \zer(s)$. We argue by contradiction.
Suppose $\zer(a) \subseteq \interi \zer(s)$. Then by Lemma~\ref{lem: GJ}, there is $f \in C(Y_A)$ such that $\zeta_A(s)=\zeta_A(a)f$ in
$C(Y_A)$. Since $C(Y_A)$ is the uniform completion of $A$ (see Proposition~\ref{prop: SW}), there is a sequence $\{ b_n \} \subseteq A$ such that $f= \lim \zeta_A(b_n)$.
It is well known that multiplication is continuous with respect to the norm, so we have $\lim \zeta_A(ab_n)= \zeta_A(a)f =\zeta_A(s)$.
Since $s \in S$, there is $\varepsilon>0$ such that $\Box s + x > \varepsilon + x$, so $(\Box s - \varepsilon) + x > 0 + x$.
There is $N$ such that $||s-ab_N||< \varepsilon$. Therefore, $s < ab_N+\varepsilon$.
Take $0 < \lambda \in \mathbb{R}$ such that $b_N \leq \lambda$. Then $s < \lambda a + \varepsilon$, so by Lemmas~\ref{lem: properties}(1) and \ref{lem: cosets}(2), and (M5),
\[
\Box s + x \leq \Box (\lambda a +\varepsilon) + x = (\Box (\lambda a) + \varepsilon) + x =  (\Box\lambda\Box a + \varepsilon) + x.
\]
But $\Box a \in x$,
so $\Box s + x \le \varepsilon + x$, contradicting $\varepsilon + x < \Box s +x$.

(3). We first show that the intersection of any two members of the family contains another member of the family. Let $s,t \in S$. Then $\Box s, \Box t \notin x$. Since $x$ is a maximal $\ell$-ideal, $A/x \cong \mathbb{R}$ is totally ordered, so
\[
(\Box s \wedge \Box t) + x = \min\{ \Box s + x, \Box t + x\} \ne 0 + x,
\]
and hence $\Box s \wedge \Box t \notin x$.
By (M1), this shows $\Box (s \wedge t) \notin x$, which gives $s \wedge t \in S$. Since $\cozer(s\wedge t)= \cozer(s) \cap \cozer(t)$, we have:
\begin{equation*}
\begin{array}{r l}
(\overline{\cozer(s)}\cap \zer(a)) \cap (\overline{\cozer(t)}\cap \zer(a)) & = \overline{\cozer(s)} \cap \overline{\cozer(t)} \cap \zer(a) \\
& \supseteq  \overline{\cozer(s) \cap \cozer(t)} \cap \zer(a)\\
& = \overline{\cozer(s\wedge t)} \cap \zer(a).
\end{array}
\end{equation*}
Because $s\wedge t \in S$, we have that $\overline{\cozer(s\wedge t)} \cap \zer(a)$ is in the family. An easy induction argument then completes the proof because every element of the family is nonempty by (2).
\end{proof}

\begin{proposition} \label{prop: difficult}
Let $(A, \Box) \in \mbal$ and $x \in Y_A$. Then $(\Box^{-1} x)^+ = \bigcup \{ y^+ \mid y \in R_\Box[x] \}$.
\end{proposition}

\begin{proof}
The right-to-left inclusion follows from the definition of $R_\Box$. For the left-to-right inclusion, let $a \in (\Box^{-1} x)^+$. By Lemma~\ref{lemmaunion}(1),
\[
\bigcap \{\cozer(s) \cap \zer(a) \mid s \in S \} = \bigcap \{ \overline{\cozer(s)} \cap \zer(a) \mid s \in S \}.
\]
By Lemma~\ref{lemmaunion}(3) and compactness of $Y_A$, this intersection is nonempty. Therefore, there is
$y \in \bigcap\{ \cozer(s) \cap \zer(a) \mid s \in S \}$. This means that
$a \in y$ and $y \cap S= \varnothing$, so $y^+ \subseteq \Box^{-1} x$. Thus, $a$ is contained in some
$y \in R_\Box[x]$, completing the proof.
\end{proof}

\begin{lemma} \label{rinvclosed}
Let $(A, \Box) \in \mbal$.
\begin{enumerate}
\item $R_\Box^{-1}[\zer(a)]=\zer(\Box a)$ for every $0\le a \in A$.
\item $D_A = \zer(\Box 0)$.
\end{enumerate}
\end{lemma}

\begin{proof}
(1). Let $x \in R_\Box^{-1}[\zer(a)]$. Then there is $y \in Y_A$ such that $x R_\Box y$ and $a \in y$. Therefore, $a \in y^+ \subseteq \Box^{-1} x$.
Thus, $\Box a \in x$, and so $x \in \zer(\Box a)$.

For the other inclusion, let $x \in \zer(\Box a)$. Since $\Box a \in x$ and $\Box a \geq 0$, we have $a \in (\Box^{-1}x)^+$. By
Proposition~\ref{prop: difficult}, there is $y \in Y_A$ such that $x R_\Box y$ and $a \in y$. Thus, $x \in R_\Box^{-1}[\zer(a)]$.

(2). This follows from (1) by setting $a = 0$ and using $Y_A = \zer(0)$.
\end{proof}

We will use Lemma~\ref{rinvclosed} to prove that $R_\Box^{-1}[F]$ is closed for each closed subset $F$ of $Y_A$. For this we require Esakia's lemma, which is an important tool in modal logic (see, e.g., \cite[Sec.~10.3]{CZ97}). The original statement is for descriptive frames,
but it has a straightforward generalization to the setting of compact Hausdorff frames (see \cite[Lem.~2.17]{BBH15}). We call
a relation $R$ on a compact Hausdorff space $X$ \emph{point-closed} if $R[x]$ is closed for each $x\in X$.

\begin{lemma} [Esakia's lemma]
If $R$ is a point-closed relation on a compact Hausdorff space $X$, then for each $($nonempty$)$ down-directed family $\{ F_i \mid i\in I \}$ of closed
subsets of $X$ we have
\[
R^{-1}\left[\bigcap\{ F_i \mid i\in I\}\right] = \bigcap \{ R^{-1}[F_i] \mid i\in I\}.
\]
\end{lemma}

\begin{remark} \label{remark: Esakia and zeros}
Let $(A,\Box)\in\mbal$ and $S$ be a set of nonnegative elements of $A$ closed under addition.
Since $\zer(a+b) \subseteq \zer(a) \cap \zer(b)$ for each $a,b \in S$, we have that
$\{ \zer(a) \mid a \in S \}$ is a down-directed family of closed subsets of $Y_A$.
Then, by Esakia's lemma and Lemma~\ref{rinvclosed}, we have:
\begin{equation*}
\begin{split}
R_\Box^{-1}[\zer(S)] & = R_\Box^{-1} \left[ \bigcap \{\zer(a) \mid a \in S \} \right] = \bigcap \{ R_\Box^{-1}[\zer(a)] \mid a \in S \}\\
& = \bigcap \{ \zer(\Box a) \mid a \in S \} = \zer(\Box S).
\end{split}
\end{equation*}
In particular, for an $\ell$-ideal $I$, since $Z_\ell(I) = Z_\ell(I^+)$, we have
\[
R_\Box^{-1} \zer(I)=R_\Box^{-1} \zer(I^+)=\bigcap \{ \zer(\Box a) \mid a \in I^+ \}.
\]
\end{remark}

\begin{proposition}\label{prop: closed}
$R_\Box^{-1}[F]$ is closed for every closed subset $F$ of $Y_A$.
\end{proposition}

\begin{proof}
Since $F$ is a closed subset of $Y_A$, there is an $\ell$-ideal $I$ such that $F=\zer(I)$.
By Remark~\ref{remark: Esakia and zeros},
\[
R_\Box^{-1} \zer(I)=\bigcap \{ \zer(\Box a) \mid a \in I^+ \},
\]
which is closed because it is an intersection of closed subsets of $Y_A$.
\end{proof}

\begin{lemma} \label{lemmadiamond}
If $\Diamond a \in x$ and $x R_\Box y$, then $a^+ \in y$.
\end{lemma}

\begin{proof}
Suppose that $x R_\Box y$ and $a^+ \notin y$. Then $a + y > 0 + y$, so there is
$0 < \lambda \in \mathbb{R}$ such that $a + y = \lambda + y$. Therefore, $\lambda-a \in y$, so $(\lambda-a)^+ \in y$.
Since $y^+ \subseteq \Box^{-1}x$, we have $(\Box(\lambda-a))^+ \in x$ by (M3), so $(\lambda + \Box (- a))^+ \in x$ by
Lemma~\ref{lem: cosets}(3). Thus,
$(\lambda + \Box(-a)) + x \le 0 + x$, so $\lambda + x \le -\Box(-a) + x$, and hence
$\lambda + x \le \Diamond a + x$ by Lemma~\ref{lem: cosets}(4).
Since $\lambda + x > 0 + x$, this shows $\Diamond a \notin x$.
\end{proof}

\begin{lemma}\label{lem: R^{-1}}
$R_\Box^{-1}[\cozer(a)]=\cozer(\Diamond a)$ for every $0 \le a \in A$.
\end{lemma}

\begin{proof}
For the left-to-right inclusion, suppose $x \notin \cozer(\Diamond a)$. Then $\Diamond a \in x$. Consider $y \in R_\Box[x]$. By
Lemma \ref{lemmadiamond}, $a=a^+ \in y$, so $y \notin \cozer (a)$. Therefore, $x \notin R_\Box^{-1}[\cozer(a)]$.

For the right-to-left inclusion, let $x \in \cozer(\Diamond a)$. Then $\Diamond a \notin x$, so $\Box 0 \in x$ by Lemma~\ref{lem: cosets}(6). Therefore, by Lemma~\ref{lem: cosets}(4), $0 + x \ne \Diamond a + x = -\Box(-a) + x$,
and hence $\Box(-a) \notin x$. Since $-a \leq 0$, we have $\Box(-a) + x \leq \Box 0 + x = 0 + x$.
Thus, there is $\lambda \in \mathbb{R}$ with $\lambda < 0$ and $\Box(-a)+x= \lambda +x$, so $\Box(-a) - \lambda \in x$.
By Lemma~\ref{lem: cosets}(3), we have
\[
\Box((-a-\lambda)^+) \in x  \mbox{ iff } (\Box(-a)-\lambda)^+ \in x.
\]
Consequently, by Proposition~\ref{prop: difficult},
\[
(-a-\lambda)^+ \in (\Box^{-1}x)^+ = \bigcup \{ y^+ \mid y \in R_\Box[x] \}.
\]
Hence, there is $y \in R_\Box[x]$ such that $(-a-\lambda)^+ \in y$. This means that $(-a-\lambda)+y \leq 0+y$, so  $a+y \geq -\lambda+y>0+y$. Therefore, $a \notin y$, and so $y \in \cozer(a)$. Thus, $x \in R_\Box^{-1}[\cozer(a)]$.
\end{proof}

\begin{proposition}\label{prop: open}
$R_\Box^{-1}[U]$ is open for every open subset $U$ of $Y_A$.
\end{proposition}

\begin{proof}
Open subsets of $Y_A$ are of the form $\cozer(I)=\bigcup \{ \cozer (a) \mid a \in I \}$ for some $\ell$-ideal $I$.
Since $\cozer(I)=\bigcup \{ \cozer (a) \mid a \in I, \, a \geq 0 \}$ and $R_\Box^{-1}$ commutes with arbitrary unions,
by Lemma~\ref{lem: R^{-1}}, we have
\begin{equation*}
\begin{split}
R_\Box^{-1}\cozer(I) & =R_\Box^{-1}\bigcup \{ \cozer (a) \mid a \in I, \, a \geq 0 \}\\
& =\bigcup \{ R_\Box^{-1}\cozer (a) \mid a \in I, \, a \geq 0 \}\\
& =\bigcup \{ \cozer (\Diamond a) \mid a \in I, \, a \geq 0 \},
\end{split}
\end{equation*}
which is open because it is a union of open subsets of $Y_A$.
\end{proof}

Putting Propositions~\ref{prop: R[x]}, \ref{prop: closed}, and \ref{prop: open}
together yields:

\begin{theorem}\label{thm: *}
If $(A,\Box)\in\mbal$, then $(Y_A,R_\Box)\in\KHK$.
\end{theorem}

We finish the section by showing how to extend the object correspondence of Theorem~\ref{thm: *} to a contravariant functor $(-)_*:\mbal\to\KHK$.

\begin{lemma} \label{propmaps2}
Let $(A,\Box),(B,\Box) \in \mbal$ and $\alpha:A \to B$ be a morphism in $\mbal$. Then $\alpha_*:Y_B \to Y_A$ is a bounded morphism.
\end{lemma}

\begin{proof}
For each $y \in Y_A$, we have that $y^+$ and $\alpha(y^+)$ are sets of nonnegative elements closed under addition, so Remark~\ref{remark: Esakia and zeros} applies. Therefore, since $Z(y^+) = \{y\}$,
\[
(\alpha_*)^{-1}(R_\Box^{-1}[y]) =(\alpha_*)^{-1}(R_\Box^{-1}[\zer(y^+)])=(\alpha_*)^{-1}(\zer(\Box y^+))
\]
and
\[
\zer(\Box \alpha(y^+))=R_\Box^{-1}[\zer(\alpha(y^+))].
\]
The definition of $\alpha_*$ shows that $(\alpha_*)^{-1}(\zer(\Box y^+))=\zer(\alpha(\Box y^+))$ and $(\alpha_*)^{-1}(\zer(y^+)) = \zer(\alpha(y^+))$. This yields
\[
(\alpha_*)^{-1}(R_\Box^{-1}[y]) = (\alpha_*)^{-1}(\zer(\Box y^+)) = \zer(\alpha(\Box y^+))
\]
and
\[
R_\Box^{-1}[(\alpha_*)^{-1}(y)] = R_\Box^{-1}[(\alpha_*)^{-1}(\zer(y^+))] = R_\Box^{-1}[\zer(\alpha(y^+))]  = \zer(\Box \alpha(y^+)).
\]
Consequently, since $\alpha$ commutes with $\Box$, we have $(\alpha_*)^{-1}(R_\Box^{-1}[y]) = R_\Box^{-1}[(\alpha_*)^{-1}(y)]$,
which proves that $\alpha_*$ is a bounded morphism.
\end{proof}

Putting Theorem~\ref{thm: *} and Lemma~\ref{propmaps2} together and remembering that $(-)_*:\bal\to\KHaus$ is a contravariant functor yields:

\begin{theorem}
$(-)_*:\mbal\to\KHK$ is a contravariant functor.
\end{theorem}

\section{Duality} \label{sec: duality}

In this section we prove our main results.
We show that $(-)_*$ and $(-)^*$ yield a dual adjunction between $\mbal$ and $\KHK$ which restricts to a dual equivalence between the category of uniformly complete members of $\mbal$ and $\KHK$.
\begin{definition}
Let $\mubal$ be the full subcategory of $\mbal$ consisting of uniformly complete objects of $\mbal$.
\end{definition}

\begin{proposition} \label{lem: reflective}
$\mubal$ is a reflective subcategory of $\mbal$.
\end{proposition}

\begin{proof}
By Proposition~\ref{prop: SW}(2), $\ubal$ is a reflective subcategory of $\bal$, where $\zeta:\bal\to\ubal$ is the reflector. We first show that $\zeta_A$ is an $\mbal$-morphism for each $(A,\Box)\in\mbal$. Let $x \in Y_A$. Recall that
\[
(\Box_{R_\Box}\zeta_A(a))(x) = \left\{\begin{array}{ll}\inf \{ \zeta_A(a)(y) \mid x R_\Box y \} & \mbox{if }x \in D_A\\ 1 & \mbox{otherwise.}\end{array}\right.
\]

If $x \in E_A$, then $\Box 0 \notin x$ by Lemma~\ref{rinvclosed}(2). Therefore, $\Box a - 1 \in x$ by Lemma~\ref{lem: cosets}(5), and hence $\zeta_A(\Box a)(x) = 1 = (\Box_{R_\Box}\zeta_A(a))(x)$. Now let $x \in D_A$. Then $(\Box_{R_\Box}\zeta_A(a))(x) = \inf \{ \zeta_A(a)(y) \mid x R_\Box y \}$. We first show that $\zeta_A(\Box a)(x) \leq \inf \{ \zeta_A(a)(y) \mid x R_\Box y \}$.
Suppose that $x R_\Box y$, so $y^+ \subseteq \Box^{-1} x$. Let $\lambda = \zeta_A(a)(y)$. Then $a - \lambda \in y$, so
$(a-\lambda)^+ \in y^+ \subseteq \Box^{-1} x$, and hence $(\Box a -\lambda)^+ \in x$ iff $\Box((a-\lambda)^+) \in x$
by Lemma~\ref{lem: cosets}(3). Therefore,
\[
0=\zeta_A((\Box a -\lambda)^+)(x)=\max \{ \zeta_A(\Box a)(x)-\lambda ,0 \},
\]
so $\zeta_A(\Box a)(x)-\lambda \leq 0$,
and hence $\zeta_A(\Box a)(x) \leq \lambda =\zeta_A(a)(y)$. Thus, $\zeta_A(\Box a)(x) \leq \inf \{ \zeta_A(a)(y) \mid x R_\Box y \}$.

We next show that $\zeta_A(\Box a)(x) \geq \inf \{ \zeta_A(a)(y) \mid x R_\Box y \}$.
Let $\mu = \zeta_A(\Box a)(x)$.
We have $\Box((a-\mu)^+) \in x$ iff
$(\Box a-\mu)^+ \in x$. Therefore, by Proposition~\ref{prop: difficult},
\[
(a-\mu)^+ \in (\Box^{-1}x)^+=\bigcup \{ y^+ \mid x R_\Box y \}.
\]
So there is $y \in R_\Box[x]$ such that $(a-\mu)^+ \in y$. Thus, $\max \{ \zeta_A(a)(y)-\mu, 0 \}=0$. This yields $\zeta_A(a)(y)-\mu \leq 0$,
and so $\zeta_A(a)(y) \leq \mu =\zeta_A(\Box a)(x)$. Consequently, $\inf \{ \zeta_A(a)(y) \mid y \in R_\Box[x] \} \leq \zeta_A(\Box a)(x)$.

Next, let $\alpha : A \to B$ be an $\mbal$-morphism with $B \in \mubal$. Since $\alpha$ is a $\bal$-morphism, there is a unique $\bal$-morphism $\gamma : C(Y_A) \to B$, given by $\gamma = \zeta_B^{-1} \circ C(\alpha_*)$, such that $\gamma \circ \zeta_A = \alpha$. 
\[
\begin{tikzcd}[column sep = 5pc]
A \arrow[r, "\zeta_A"] \arrow[d, "\alpha"'] & C(Y_A) \arrow[d, "C(\alpha_*)"] \arrow[ld, "\gamma"] \\
B & \arrow[l, "\zeta_B^{-1}"] C(Y_B)
\end{tikzcd}
\]
As we saw in the paragraph above, $\zeta_B$ is an $\mbal$-morphism. Also, $C(\alpha_*) : C(Y_A) \to C(Y_B)$ is an $\mbal$-morphism by Lemmas~\ref{propmaps2} and \ref{lem: varphi* a morphism}. Therefore, $\gamma$ is an $\mbal$-morphism, concluding the proof.
\end{proof}

\begin{theorem}\label{thm: mbal}
The functors $(-)_*:\mbal\to\KHK$ and $(-)^*:\KHK\to\mbal$ yield a dual adjunction of the categories, which restricts to
a dual equivalence between $\mubal$ and $\KHK$.
\[
\begin{tikzcd}
\mubal \arrow[rr, hookrightarrow] && \mbal \arrow[dl, "(-)_*"]  \arrow[ll, bend right = 20] \\
&  \KHK \arrow[ul,  "(-)^*"] &
\end{tikzcd}
\]
\end{theorem}

\begin{proof}
By Gelfand duality, the functors $(-)_*:\bal\to\KHaus$ and $(-)^*:\KHaus\to\bal$ yield a dual adjunction between $\bal$ and
$\KHaus$ that restricts to a dual equivalence between $\ubal$ and $\KHaus$. The natural transformations are given by $\zeta:1_{\bal}\to(-)^* \circ (-)_*$ and
$\varepsilon:1_{\KHaus}\to(-)_* \circ (-)^*$ where we recall from Section~\ref{sec: Gelfand duality} that $\varepsilon_X: X \to X_{C(X)}$ is defined by
\[
\varepsilon_X(x) = M_x = \{f \in C(X) \mid f(x)=0 \}.
\]
Therefore, it is sufficient to show that $\zeta_A$ is a morphism in $\mbal$ for each $(A, \Box) \in \mbal$ and that $\varepsilon_X$ is a bounded morphism for each $(X, R) \in \KHK$. We showed in the proof of Proposition~\ref{lem: reflective} that $\zeta_A(\Box a)=\Box_{R_\Box}\zeta_A(a)$ for each $(A,\Box)\in\mbal$ and $a\in A$. Thus, $\zeta_A$ is a morphism in $\mbal$, and hence it remains to show that $xRy$ iff $\varepsilon_X(x) R_{\Box_R} \varepsilon_X(y)$ for each $(X,R)\in\KHK$.

To see this recall that
$\varepsilon_X(x) R_{\Box_R} \varepsilon_X(y)$ means that $M_y^+ \subseteq \Box_R^{-1} M_x$. First suppose that $x R y$ and $f \in M_y^+$. Then
$f(y)=0$ and $f \geq 0$. We have $(\Box_R f) (x)= \inf \{ f(z) \mid x R z \}=0$. Therefore, $\Box_R f \in M_x$, and so $f \in \Box_R^{-1} M_x$.
This gives $M_y^+ \subseteq \Box_R^{-1} M_x$.
Next suppose that $x {\xnot R} y$, so $y \notin R[x]$. If $R[x] = \varnothing$, then $(\Box_R0)(x) = 1$, so $0 \in M_y^+$ but $\Box_R 0 \notin M_x$, yielding $M_y^+ \not\subseteq \Box_R^{-1} M_x$. On the other hand, if $R[x] \ne \varnothing$, since $R[x]$ is closed, by Urysohn's Lemma there is $f \geq 0$ such that $f(y)=0$ and $f(R[x]) = \{1\}$. Thus, $f \in M_y^+$ and
$\Box_R f \notin M_x$. Consequently, $M_y^+ \nsubseteq \Box_R^{-1}M_x$.
\end{proof}

\section{Connections with modal algebras and descriptive frames}\label{sec: connection with Esakia-Goldblatt}

In this section we connect Theorem~\ref{thm: mbal} to the duality between $\ma$ and $\DF$.
Recall that a \emph{modal algebra} is a pair $\mathfrak A=(A,\Box)$ where $A$ is a boolean algebra and $\Box$ is a unary function on $A$ preserving finite meets (including 1). As usual, the dual function $\Diamond$ is defined by $\Diamond a=\neg\Box\neg a$, and is axiomatized as a unary function preserving
finite joins (including 0). Let $\ma$ be the category of modal algebras and modal homomorphisms (boolean homomorphisms preserving $\Box$).

We recall from the Introduction that a \emph{descriptive frame} is a pair $\mathfrak F=(X,R)$ where $X$ is a Stone space and $R$ is a continuous relation on $X$, and that $\DF$ is the category of descriptive frames and continuous bounded morphisms.
As we already pointed out,
Stone duality generalizes to the following duality:

\begin{theorem} [\cite{Esa74,Gol76}] \label{thm: EG-theorem}
$\ma$ is dually equivalent to $\DF$.
\end{theorem}

The functors $(-)^*:\DF\to\ma$ and $(-)_*:\ma\to\DF$
are defined as follows. For a descriptive Kripke frame $\mathfrak F=(X,R)$ let $\mathfrak F^*=(\Clop(X),\Box_R)$ where $\Clop(X)$ is the boolean algebra of clopen subsets of $X$ and $\Box_R U = X \setminus R^{-1}[X \setminus U]$ (alternatively, $\Diamond_R U = R^{-1}[U]$). For a bounded morphism $f$ let $f^*=f^{-1}$. Then $(-)^*:\DF\to\ma$ is a well-defined contravariant functor.

For a modal algebra $\mathfrak A=(A,\Box)$ let $\mathfrak A_*=(Y_A,R_\Box)$ where $Y_A$ is the set of ultrafilters of $A$ and
\[
x R_\Box y \quad\mbox{iff}\quad (\forall a\in A)(\Box a \in x \Rightarrow a \in y) \quad\mbox{iff}\quad \Box^{-1}x \subseteq y
\]
(alternatively, $x R_\Box y$ iff $(\forall a\in A)(a \in y \Rightarrow \Diamond a\in x)$ iff $y \subseteq \Diamond^{-1}x$). For a modal algebra homomorphism $h$ let $h_*=h^{-1}$. Then $(-)_*:\ma\to\DF$ is a well-defined contravariant functor, and the functors $(-)_*$ and $(-)^*$ yield a dual equivalence of $\ma$ and $\DF$.

To define a functor from $\mbal$ to $\ma$ we recall that for each commutative ring $A$ with $1$, the idempotents of $A$ form a boolean algebra $\Id(A)$, where the boolean
operations on $\Id(A)$ are defined as follows:
\[
e \wedge f = ef, \ \ \ e \vee f = e + f - ef, \ \ \ \neg e = 1 - e.
\]
We point out that if $A \in \bal$, then the lattice operations on $A$ restrict to those on $\Id(A)$.

\begin{remark} \label{rem: Birkhoff}
We will use the following two identities of $f$-rings (see \cite[Sec.~XIII.3]{Bir79} and \cite[Cor.~XVII.5.1]{Bir79}):
\[
(a \wedge b) + c = (a + c) \wedge (b + c) \quad\mbox{and}\quad (a \wedge b) d = (ad) \wedge (bd) \mbox{ for } d \ge 0.
\]
\end{remark}

\begin{lemma} \label{lem: Box on Id(A)}
If $(A,\Box)\in\mbal$, then $\Box$ sends idempotents to idempotents.
\end{lemma}

\begin{proof}
First observe that $e \in A$ is an idempotent iff $1 \wedge 2e = e$. To see this, if $e$ is an idempotent, by Remark~\ref{rem: Birkhoff},
\[
(1 \wedge 2e) - e = (1-e) \wedge e = \lnot e \wedge e = 0.
\]
Therefore, $1 \wedge 2e = e$. Conversely, suppose that $1 \wedge 2e = e$. Then $(1-e) \wedge e = 0$ by the same calculation. Since each
$A \in \bal$ is an $f$-ring (see, e.g., \cite[Lem.~XVII.5.2]{Bir79}), from $(1-e) \wedge e = 0$ it follows that $(1-e)e = 0$ (see, e.g.,
\cite[Lem.~XVII.5.1]{Bir79}). Thus, $e^2 = e$, and hence $e$ is an idempotent.

For each $a \in A$, by (M5), (M2), and Lemma~\ref{lem: properties}(4) we have
\[
\Box (2 a) = \Box 2 \Box a=(2-\Box 0)\Box a=(2-2\Box 0+\Box 0)\Box a=2 \Box a (1-\Box 0)+\Box 0.
\]
By Lemma~\ref{lem: properties}(3), $\Box 0 \ge 0$, so Lemma~\ref{lem: properties}(4) and Remark~\ref{rem: Birkhoff} imply
\[
(1 \wedge 2 \Box a)\Box 0=\Box 0 \wedge 2\Box a \Box 0=\Box 0 \wedge 2\Box 0=\Box 0.
\]
Now suppose $e$ is an idempotent, so $e = 1 \wedge 2e$. Since  $\Box 0\le\Box 1=1$, we have $1-\Box 0 \ge 0$.  Thus, by
Remark~\ref{rem: Birkhoff} and the two identities just proved,
\begin{align*}
\Box e &= \Box (1 \wedge 2e)=1 \wedge \Box(2e)  \\
&= ((1-\Box 0)+\Box 0) \wedge \Box(2e) \\
&= ((1-\Box 0)+\Box 0) \wedge (2 \Box e (1-\Box 0)+\Box 0) \\
&= ((1-\Box 0) \wedge 2 \Box e (1-\Box 0))+\Box 0 \\
&= (1 \wedge 2 \Box e)(1-\Box 0)+\Box 0 \\
&= (1 \wedge 2 \Box e)(1-\Box 0)+(1 \wedge 2 \Box e)\Box 0 \\
&= 1 \wedge 2 \Box e.
\end{align*}
Therefore, $\Box e$ is idempotent.
\end{proof}

\begin{lemma}\label{lem: idempotents}
If $(A,\Box)\in\mbal$, then $(\Id(A),\Box)\in\ma$.
\end{lemma}

\begin{proof}
Since $A\in\bal$, we have that $\Id(A)$ is a boolean algebra.
By Lemma~\ref{lem: Box on Id(A)}, $\Box$ is well defined on $\Id(A)$. That $\Box$ preserves finite meets in $\Id(A)$ follows from (M1) and
Lemma~\ref{lem: properties}(2). Thus, $(\Id(A),\Box)\in\ma$.
\end{proof}

Define $\Id:\mbal\to\ma$ by sending $(A,\Box)\in\mbal$ to $(\Id(A),\Box)\in\ma$ and a morphism $A\to B$ in $\mbal$ to its restriction
$\Id(A)\to\Id(B)$. The next lemma is an easy consequence of Lemma~\ref{lem: idempotents}.

\begin{lemma}\label{lem: Id}
$\Id:\mbal\to\ma$ is a well-defined covariant functor.
\end{lemma}

We recall (see \cite{McG06} and the references therein) that a commutative ring $A$ is \emph{clean} if each element is the sum of an idempotent and a unit.

\begin{definition}
Let $\cubal$ be the full subcategory of $\ubal$ consisting of those $A \in \ubal$ where $A$ is clean.
\end{definition}

\begin{remark}\label{rem: cubal}
By Stone duality for boolean algebras and \cite[Prop.~5.20]{BMO13a}, the following diagram commutes (up to natural isomorphism),
and the functor $\Id$ yields an equivalence of $\cubal$ and $\ba$.
\[
\begin{tikzcd}[row sep = 3pc]
\cubal
\arrow[rr, "\Id"] \arrow[dr, shift left = .5ex, "(-)_*"] && \ba \arrow[dl, shift left = .5ex, "(-)_*"] \\
& \Stone \arrow[ur, shift left = .5ex, "(-)^*"] \arrow[lu, shift left = .5ex, "(-)^*"]&
\end{tikzcd}
\]
\end{remark}

\begin{definition}
Let $\mcubal$ be the full subcategory of $\mubal$ consisting of those $(A,\Box) \in \mubal$ where $A$ is clean.
\end{definition}

As a corollary of Theorems~\ref{thm: mbal}, \ref{thm: EG-theorem} and Remark~\ref{rem: cubal}, we obtain:

\begin{theorem}\label{thm: clean}
The diagram below commutes $($up to natural isomorphism$)$ and the functor $\Id$ yields an equivalence of $\mcubal$ and $\ma$.
\[
\begin{tikzcd}[row sep = 3pc]
\mcubal
\arrow[rr, "\Id"] \arrow[dr, shift left = .5ex, "(-)_*"] && \ma \arrow[dl, shift left = .5ex, "(-)_*"] \\
&\DF \arrow[ur, shift left = .5ex, "(-)^*"] \arrow[lu, shift left = .5ex, "(-)^*"] &
\end{tikzcd}
\]
\end{theorem}

\section{Some correspondence results} \label{subsec: correspondence}

In this section we take the first steps towards the correspondence theory for $\mbal$ by characterizing algebraically what it takes for the
relation $R_\Box$ to satisfy additional first-order properties, such as seriality, reflexivity, transitivity, and symmetry.

\begin{convention}
To simplify notation, we denote the dual $(Y_A, R_\Box)$ of $(A, \Box) \in \mbal$ simply by $(Y, R)$.
\end{convention}

We recall that a relation $R$ on $X$ is \emph{serial} if $R[x] \ne \varnothing$ for each $x \in X$.

\begin{proposition}\label{prop: correspondence for (A, Box) serial}
Let $(A, \Box) \in \mbal$. Then $R$ is serial iff $\Box 0=0$ in $A$.
\end{proposition}

\begin{proof}
Suppose that $R$ is serial. Then $R[x] \ne \varnothing$, so $(\Box_{R} 0)(x)=  0$ for each $x \in Y$. Thus, $\Box_{R} 0 = 0$. Since $(A, \Box)$ embeds into $(C(Y),\Box_{R})$, we conclude that $\Box 0=0$ in $A$.
Conversely, suppose that $\Box 0 = 0$ in $A$. Since $Y=\zer(0)$, by Lemma~\ref{rinvclosed}(2), we have $D_A=\zer(\Box 0)=\zer(0)=Y$.
Thus, $R$ is serial.
\end{proof}

\begin{proposition}\label{prop: correspondence for (A, Box) reflexive}
Let $(A, \Box) \in \mbal$. Then $R$ is reflexive iff $\Box a \le a$ for each $a \in A$.
\end{proposition}

\begin{proof}
Suppose that $R$ is reflexive and $f \in C(Y)$. For each $x \in Y$, we have $x \in R[x]$. Thus, $(\Box_{R} f)(x) = \inf fR[x] \le f(x)$.  Since $(A, \Box)$ embeds into $(C(Y),\Box_{R})$, we conclude that $\Box a \le a$ for each $a \in A$.
Conversely, suppose $\Box a \le a$ for each $a\in A$. Let $x \in Y$ and $a \in x^+$. Then $0 \le \Box a \le a \in x$. Thus, $x^+ \subseteq \Box^{-1} x$,
and so $x R x$.
\end{proof}

\begin{proposition}\label{prop: correspondence for (A, Box) transitive}
Let $(A, \Box) \in \mbal$. Then $R$ is transitive iff $\Box a \le \Box(\Box a(1-\Box 0)+a\Box 0)$ for each $a\in A$.
\end{proposition}

\begin{proof}
First suppose that $R$ is transitive. Let $f \in C(Y)$ and $x \in Y$. If $R[x] = \varnothing$, then by definition of $\Box_R$
\[
(\Box_R f)(x)=1=\Box_R(\Box_R f(1-\Box_R 0)+f\Box_R 0)(x).
\]
Suppose that $R[x] \ne \varnothing$. Then $(\Box_R f)(x)= \inf fR[x]$ and
\[
\Box_R(\Box_R f(1-\Box_R 0)+f\Box_R 0)(x)= \inf \{(\Box_R f)(y)(1-\Box_R 0)(y)+f(y)(\Box_R 0)(y) \mid xRy \}.
\]
We have
\[
(\Box_R f)(y)(1-\Box_R 0)(y)+f(y)(\Box_R 0)(y) = \left\{\begin{array}{ll} f(y) & \mbox{if }R[y]=\varnothing\\ (\Box_R f)(y) &
\mbox{if }R[y]\ne \varnothing.\end{array}\right.
\]
It is therefore sufficient to prove that, for each $y \in R[x]$, if $R[y]=\varnothing$ then $(\Box_R f)(x) \le f(y)$ and if
$R[y] \ne \varnothing$ then $(\Box_R f)(x) \le (\Box_R f)(y)$. Suppose $R[y]=\varnothing$. Since $R[x] \ne \varnothing$, we have
\[
(\Box_R f)(x)= \inf \{ f(z) \mid z \in R[x] \} \le f(y).
\]
If $R[y] \ne \varnothing$, then by transitivity of $R$ we have $R[y] \subseteq R[x]$, so
\[
(\Box_R f)(x)= \inf \{ f(z) \mid z \in R[x] \} \le \inf \{ f(w) \mid w \in R[y] \} = (\Box_R f)(y).
\]
Thus, $\Box_R f \le \Box_R(\Box_R f(1-\Box_R 0)+f\Box_R 0)$.
Since $(A, \Box)$ embeds into $(C(Y),\Box_{R})$, we conclude that $\Box a \le \Box(\Box a(1-\Box 0)+a\Box 0)$ for each $a\in A$.

Conversely, suppose $\Box a \le \Box(\Box a(1-\Box 0)+a\Box 0)$ for each $a\in A$. Let $x,y,z \in Y$ with $x R y$ and $y R z$. Then $y^+ \subseteq \Box^{-1} x$ and $z^+ \subseteq \Box^{-1} y$. Suppose that $a \in z^+$. Then $\Box a \in y^+$. Since $0 \in z^+$, we have $\Box 0 \in y^+$. Thus, since $y$ is an ideal, $\Box a(1-\Box 0)+a \Box 0 \in y$. Because $\Box a(1-\Box 0)+a \Box 0 \ge 0$, we have $\Box(\Box a(1-\Box 0)+a \Box 0) \in x$. By hypothesis,
$0 \le \Box a \le \Box(\Box a(1-\Box 0)+a\Box 0) \in x$. Thus, $\Box a \in x$. This shows that $z^+ \subseteq \Box^{-1} x$, and hence $x R z$.
\end{proof}

\begin{proposition}\label{prop: correspondence for (A, Box) symmetric}
Let $(A, \Box) \in \mbal$. Then $R$ is symmetric iff $\Diamond \Box a(1-\Box 0) \le a(1-\Box 0) $ for each $a\in A$.
\end{proposition}

\begin{proof}
First suppose that $R$ is symmetric. Let $f \in C(Y)$ and $x \in Y$. If $R[x] = \varnothing$, then $(1-\Box_R 0)(x)=0$ so
\[
(\Diamond_R \Box_R f)(x)(1-\Box_R 0)(x) = 0 = f(x)(1-\Box_R 0)(x).
\]
If $R[x] \ne \varnothing$, then $(1-\Box_R 0)(x)=1$, so it is sufficient to prove that $(\Diamond_R \Box_R f)(x)\le f(x)$.
For any $y \in R[x]$ we have $x \in R[y]$ by symmetry. Therefore,
\[
(\Box_R f)(y) =  \inf\{f(z) \mid z \in R[y] \} \le f(x).
\]
Thus, recalling Remark~\ref{rem: sup}, we have
\[
(\Diamond_R \Box_R f)(x)  = \sup \{(\Box_R f)(y) \mid y \in R[x] \} \le f(x).
\]
Since $(A, \Box)$ embeds into $(C(Y),\Box_{R})$, we conclude that $\Diamond \Box a(1-\Box 0) \le a(1-\Box 0) $ for each $a\in A$.

Conversely, suppose $\Diamond \Box a(1-\Box 0) \le a(1-\Box 0)$ for each $a\in A$. Let $x,y \in Y$ with $x R y$. Then $y^+ \subseteq \Box^{-1} x$, so $0 \in y^+$ implies $\Box 0 \in x$. Thus,
\[
\Diamond \Box a+ x = \Diamond \Box a(1-\Box 0) + x \le  a(1-\Box 0) + x = a + x.
\]
To see that $y R x$, let $a \in x^+$. If $\Box a \notin y$, then
$0+y < \Box a +y$ because $\Box a \ge 0$. So there is $0 < \lambda \in \mathbb{R}$ such that $\lambda - \Box a \in y$.
Thus, $(\lambda - \Box a)^+ \in y^+$. Since $x R y$, by (2) and (4) of Lemma~\ref{lem: cosets}, we have
\[
(\lambda-\Diamond \Box a)^+ +x= (\lambda + \Box (- \Box a))^+ +x=(\Box (\lambda - \Box a))^+ + x=\Box (\lambda - \Box a)^+ +x=0+x.
\]
Because $\Diamond \Box a +x \le a+x$ we have $(\lambda - a) +x \le (\lambda - \Diamond\Box a) +x $. Therefore,
\[
0 \le (\lambda-a)^+ + x \le (\lambda-\Diamond \Box a)^+ +x =0+x.
\]
This implies $(\lambda-a)^+ \in x$. Thus, by Remark~\ref{rem: properties of primes}(4), $0+x < \lambda +x \le a+x$, which contradicts $a \in x^+$.
Therefore, $\Box a \in y$, which yields $x^+ \subseteq \Box^{-1} y$. Thus, $y R x$.
\end{proof}


\begin{remark}
If we work with $\Diamond$ instead of $\Box$, then Propositions~\ref{prop: correspondence for (A, Box) serial}---\ref{prop: correspondence for (A, Box) symmetric}
can be stated as follows.
\begin{enumerate}
\item $R$ is serial iff $\Diamond 1 = 1$.
\item $R$ is reflexive iff $a \le \Diamond a$ for each $a \in A$.
\item $R$ is transitive iff $\Diamond(\Diamond a +a(1-\Diamond 1)) \le \Diamond a$ for each $a \in A$.
\item $R$ is symmetric iff $\Diamond \Box a \le  a\Diamond 1$ for each $a \in A$.
\end{enumerate}
\end{remark}

\begin{remark} \label{rem: serial 1}
Let $(A, \Box) \in \mbal$. If $\Box 0 = 0$, then
the transitivity and symmetry axioms simplify to $\Box a \le \Box \Box a$ and $\Diamond \Box a \le a$, which are standard transitivity
and symmetry axioms in modal logic.
\end{remark}

\begin{definition}\label{def: notation}
\begin{enumerate}
\item[]
\item Let $\mbal^{\sf D}$ be the full subcategory of $\mbal$ consisting of objects $(A, \Box) \in \mbal$ satisfying $\Box 0 = 0$.
\item Let $\mbal^{\sf T}$ be the full subcategory of $\mbal$ consisting of objects $(A, \Box) \in \mbal$ satisfying $\Box a \le a$.
\item Let $\mbal^{\sf K4}$ be the full subcategory of $\mbal$ consisting of objects $(A, \Box) \in \mbal$ satisfying
$\Box a \le \Box(\Box a(1-\Box 0)+a\Box 0)$.
\item Let $\mbal^{\sf B}$ be the full subcategory of $\mbal$ consisting of objects $(A, \Box) \in \mbal$ satisfying
$\Diamond \Box a(1-\Box 0) \le a(1-\Box 0)$.
\item Let $\mbal^{\sf S4} = \mbal^{\sf T} \cap \mbal^{\sf K4}$.
\item Let $\mbal^{\sf S5} = \mbal^{\sf S4} \cap \mbal^{\sf B}$.
\end{enumerate}
\end{definition}

\begin{remark}
Since the reflexivity axiom implies the seriality axiom, we obtain that $(A, \Box) \in \mbal^{\sf S4}$ iff $(A, \Box) \in \mbal^{\sf T}$ and
$\Box a \le \Box \Box a$ for each $a \in A$. Similarly, $(A, \Box) \in \mbal^{\sf S5}$ iff $(A, \Box) \in \mbal^{\sf S4}$ and
$\Diamond\Box a \le a$ for each $a \in A$.
\end{remark}

\begin{remark}
The notation of Definition~\ref{def: notation} is motivated by the standard notation in modal logic:
\begin{enumerate}
\item $\sf D$ denotes the least normal modal logic containing the axiom $\Diamond \top$.
\item $\sf T$ denotes the least normal modal logic containing the axiom $\Box p\to p$.
\item $\sf K4$ denotes the least normal modal logic containing the axiom $\Box p\to\Box\Box p$.
\item $\sf B$ denotes the least normal modal logic containing the axiom $\Diamond \Box p \to p$.
\item $\sf S4$ denotes the join of $\sf T$ and $\sf K4$.
\item $\sf S5$ denotes the join of $\sf S4$ and $\sf B$.
\end{enumerate}
\end{remark}

As with the corresponding classes of modal algebras, we have the following inclusions between the subclasses of algebras in $\mbal$ given in
Definition~\ref{def: notation}:

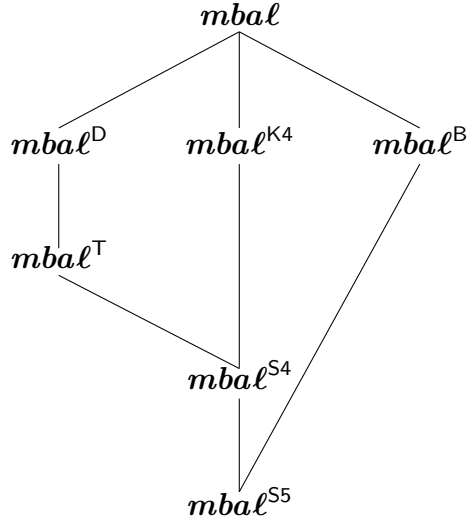
\bigskip

\begin{center}
\begin{tikzpicture}[scale = .80]
\node at (7,8.1) {$\mbal$};
\draw (7,7.8) -- (4,6.2);
\draw (7,7.8) -- (7,6.2);
\draw (7,7.8) -- (10,6.2);
\node  at (4,6) {$\mbal^{\sf D}$};
\draw (4,5.6) -- (4,4.2);
\node at (7,6) {$\mbal^{\sf K4}$};
\draw (7,5.6) -- (7,2.2);
\node  at (10,6) {$\mbal^{\sf B}$};
\draw (10,5.6) -- (7,0.15);
\node  at (4,4.1) {$\mbal^{\sf T}$};
\draw (4,3.75) -- (7,2.2);
\node  at (7,2.1) {$\mbal^{\sf S4}$};
\draw (7,1.7) -- (7,0.15);
\node  at (7,0) {$\mbal^{\sf S5}$};
\end{tikzpicture}
\end{center}

Similarly to Definition~\ref{def: notation}, for $\sf X \in \{ D, T, K4, B, S4, S5 \}$ we define the following categories:

\begin{itemize}
\item The categories $\mubal^{\sf X}$ are defined similarly to $\mbal^{\sf X}$ but with $\mbal$ replaced by $\mubal$.
\item The categories $\mcubal^{\sf X}$ are defined similarly to $\mbal^{\sf X}$ but with $\mbal$ replaced by $\mcubal$.
\item The categories $\ma^{\sf X}$ are defined similarly to $\mbal^{\sf X}$ but with $\mbal$ replaced by $\ma$.
\item The categories $\KHK^{\sf X}$ are defined by adding the corresponding properties on the relation $R$ to the definition of $\KHK$.
\item The categories $\DF^{\sf X}$ are defined as $\KHK^{\sf X}$ by restricting $\KHK$ to $\DF$.
\end{itemize}

Theorems~\ref{thm: mbal} and~\ref{thm: clean},
Propositions~\ref{prop: correspondence for (A, Box) serial}---\ref{prop: correspondence for (A, Box) symmetric}, and the corresponding versions of Theorem~\ref{thm: EG-theorem} yield the following result.

\begin{theorem}
Suppose that $\sf X \in \{ D, T, K4, B, S4, S5 \}$.
\begin{enumerate}
\item The category $\mubal^{\sf X}$ is dually equivalent to $\KHK^{\sf X}$.
\item The categories $\mcubal^{\sf X}$ and $\ma^{\sf X}$ are dually equivalent to $\DF^{\sf X}$, and hence are equivalent.
\end{enumerate}
\end{theorem}

\section{Concluding Remarks} \label{sec: conclusion}

We finish the paper with several remarks, which indicate a number of possible directions for future research.

\begin{remark} \label{rem: concluding remarks}
\begin{enumerate}
\item[]

\item As we pointed out in the Introduction, there are other dualities for $\KHaus$. For example, in pointfree topology we have Isbell duality \cite{Isb72} (see also \cite{BM80} or \cite[Sec.~III.1]{Joh82}) and de Vries duality \cite{deV62} (see also \cite{Bez10}). The two are closely related,
see \cite{Bez12}. Isbell and de Vries dualities were generalized to the setting of $\KHK$ in \cite{BBH15}. We plan to compare the results of \cite{BBH15} to the ones obtained in this paper.

\item Another relevant duality was established by Kakutani \cite{Kak40,Kak41}, the Krein brothers \cite{KK40}, and Yosida \cite{Yos41}, who also worked with continuous real-valued functions, but their signature was that of a vector lattice instead of an $\ell$-algebra.
Gelfand duality has a natural counterpart in this setting. Let $\bav$ be the category of bounded archimedean vector lattices and let $\ubav$ be its reflective subcategory consisting of uniformly complete objects. Then there is a dual adjunction between $\bav$ and $\KHaus$, which restricts to a dual equivalence between $\ubav$ and $\KHaus$. This duality is known as Yosida duality (or Kakutani-Krein-Yosida duality). In our axiomatization of $\mbal$ (see Definition~\ref{def:mbal}), the only axiom involving multiplication
is (M5). In the serial case (M5) simplifies to (M5${}'$) of Remark~\ref{rem: serial}, which only involves scalar multiplication. In the non-serial
case, (M5) can be replaced by the following two axioms
\begin{itemize}
\item $\Box(\lambda a) = \lambda\Box a + (1-\lambda) \Box 0$ provided $\lambda \ge 0$,
\item $\Box 0 \wedge (1-\Box a)^+ = 0$, 
\end{itemize}
which again only involve vector lattice operations. This yields the category $\mbav$ of modal bounded archimedean vector lattices and its reflective subcategory $\mubav$ consisting of uniformly complete objects. The results of
Section~\ref{sec: duality} then generalize to the setting of $\mbav$ and $\mubav$, and provide a generalization of Yosida duality.

\item Our definition of a modal operator on a bounded archimedean $\ell$-algebra can be further adjusted to the settings of $\ell$-rings,
$\ell$-groups, and MV-algebras.
In this regard, it would be interesting to develop logical systems corresponding to these algebras.

\item It would be natural to develop the correspondence theory for $\mbal$ by generalizing the results of Section~\ref{subsec: correspondence},
with the final goal towards a Sahlqvist type correspondence (see, e.g., \cite[Ch.~3]{BRV01}).
\end{enumerate}
\end{remark}

\def\cprime{$'$}
\providecommand{\bysame}{\leavevmode\hbox to3em{\hrulefill}\thinspace}
\providecommand{\MR}{\relax\ifhmode\unskip\space\fi MR }
\providecommand{\MRhref}[2]{%
  \href{http://www.ams.org/mathscinet-getitem?mr=#1}{#2}
}
\providecommand{\href}[2]{#2}

\bigskip

Department of Mathematical Sciences, New Mexico State University, Las Cruces NM 88003,
guram@nmsu.edu, lcarai@nmsu.edu, pmorandi@nmsu.edu


\begin{thebibliography}{10}

\bibitem{BM80}
B.~Banaschewski and C.~J. Mulvey, \emph{Stone-\v {C}ech compactification of
  locales. {I}}, Houston J. Math. \textbf{6} (1980), no.~3, 301--312.

\bibitem{Bez10}
G.~Bezhanishvili, \emph{Stone duality and {G}leason covers through de {V}ries
  duality}, Topology Appl. \textbf{157} (2010), no.~6, 1064--1080.

\bibitem{Bez12}
\bysame, \emph{De {V}ries algebras and compact regular frames}, Appl. Categ.
  Structures \textbf{20} (2012), no.~6, 569--582.

\bibitem{BBH15}
G.~Bezhanishvili, N.~Bezhanishvili, and J.~Harding, \emph{Modal compact
  {H}ausdorff spaces}, J. Logic Comput. \textbf{25} (2015), no.~1, 1--35.

\bibitem{BMO13a}
G.~Bezhanishvili, P.~J. Morandi, and B.~Olberding, \emph{Bounded {A}rchimedean
  {$\ell$}-algebras and {G}elfand-{N}eumark-{S}tone duality}, Theory Appl.
  Categ. \textbf{28} (2013), Paper No. 16, 435--475.

\bibitem{BMO16}
\bysame, \emph{A functional approach to {D}edekind completions and the
  representation of vector lattices and {$\ell$}-algebras by normal functions},
  Theory Appl. Categ. \textbf{31} (2016), Paper No. 37, 1095--1133.

\bibitem{Bir79}
G.~Birkhoff, \emph{Lattice theory}, third ed., American Mathematical Society
  Colloquium Publications, vol.~25, American Mathematical Society, Providence,
  R.I., 1979.

\bibitem{BRV01}
P.~Blackburn, M.~de~Rijke, and Y.~Venema, \emph{Modal logic}, Cambridge Tracts
  in Theoretical Computer Science, vol.~53, Cambridge University Press,
  Cambridge, 2001.

\bibitem{CZ97}
A.~Chagrov and M.~Zakharyaschev, \emph{Modal logic}, Oxford Logic Guides,
  vol.~35, The Clarendon Press, Oxford University Press, New York, 1997.

\bibitem{deV62}
H.~de~Vries, \emph{Compact spaces and compactifications. {A}n algebraic
  approach}, Ph.D. thesis, University of Amsterdam, 1962.

\bibitem{Esa74}
L.~L. Esakia, \emph{Topological {K}ripke models}, Dokl. Akad. Nauk SSSR
  \textbf{214} (1974), 298--301.

\bibitem{GN43}
I.~Gelfand and M.~Neumark, \emph{On the imbedding of normed rings into the ring
  of operators in {H}ilbert space}, Rec. Math. [Mat. Sbornik] N.S.
  \textbf{12(54)} (1943), 197--213.

\bibitem{GJ60}
L.~Gillman and M.~Jerison, \emph{Rings of continuous functions}, The University
  Series in Higher Mathematics, D. Van Nostrand Co., Inc., Princeton,
  N.J.-Toronto-London-New York, 1960.

\bibitem{Gol76}
R.~I. Goldblatt, \emph{Metamathematics of modal logic}, Rep. Math. Logic
  (1976), no.~6, 41--77.

\bibitem{Hal56}
P.~R. Halmos, \emph{Algebraic logic. {I}. {M}onadic {B}oolean algebras},
  Compositio Math. \textbf{12} (1956), 217--249.

\bibitem{HJ61}
M.~Henriksen and D.~G. Johnson, \emph{On the structure of a class of
  {A}rchimedean lattice-ordered algebras}, Fund. Math. \textbf{50} (1961/1962),
  73--94.

\bibitem{Isb72}
J.~Isbell, \emph{Atomless parts of spaces}, Math. Scand. \textbf{31} (1972),
  5--32.

\bibitem{Joh82}
P.~T. Johnstone, \emph{Stone spaces}, Cambridge Studies in Advanced
  Mathematics, vol.~3, Cambridge University Press, Cambridge, 1982.

\bibitem{JT51}
B.~J{\'o}nsson and A.~Tarski, \emph{Boolean algebras with operators. {I}},
  Amer. J. Math. \textbf{73} (1951), 891--939.

\bibitem{Kak40}
S.~Kakutani, \emph{Weak topology, bicompact set and the principle of duality},
  Proc. Imp. Acad. Tokyo \textbf{16} (1940), 63--67.

\bibitem{Kak41}
\bysame, \emph{Concrete representation of abstract {$(M)$}-spaces. ({A}
  characterization of the space of continuous functions.)}, Ann. of Math. (2)
  \textbf{42} (1941), 994--1024.

\bibitem{Kra99}
M.~Kracht, \emph{Tools and techniques in modal logic}, Studies in Logic and the
  Foundations of Mathematics, vol. 142, North-Holland Publishing Co.,
  Amsterdam, 1999.

\bibitem{KK40}
M.~Krein and S.~Krein, \emph{On an inner characteristic of the set of all
  continuous functions defined on a bicompact {H}ausdorff space}, C. R.
  (Doklady) Acad. Sci. URSS (N.S.) \textbf{27} (1940), 427--430.

\bibitem{Kri63}
S.~A. Kripke, \emph{Semantical considerations on modal logic}, Acta Philos.
  Fenn. \textbf{Fasc.} (1963), 83--94.

\bibitem{KKV04}
C.~Kupke, A.~Kurz, and Y.~Venema, \emph{Stone coalgebras}, Theoret. Comput.
  Sci. \textbf{327} (2004), no.~1-2, 109--134.

\bibitem{McG06}
W.~W. McGovern, \emph{Neat rings}, J. Pure Appl. Algebra \textbf{205} (2006),
  no.~2, 243--265.

\bibitem{SV88}
G.~Sambin and V.~Vaccaro, \emph{Topology and duality in modal logic}, Ann. Pure
  Appl. Logic \textbf{37} (1988), no.~3, 249--296.

\bibitem{Sto40}
M.~H. Stone, \emph{A general theory of spectra. {I}}, Proc. Nat. Acad. Sci.
  U.S.A. \textbf{26} (1940), 280--283.

\bibitem{Yos41}
K.~Yosida, \emph{On vector lattice with a unit}, Proc. Imp. Acad. Tokyo
  \textbf{17} (1941), 121--124.

\end{thebibliography}
\end{document}